\documentclass{article}
\usepackage{amsmath}
\usepackage{amssymb}
\usepackage{amsthm}
\usepackage{amsfonts}
\usepackage{latexsym}
\usepackage{verbatim}   
\usepackage{color}    
\usepackage{mathrsfs,epsfig}  
\usepackage{fontenc}
\usepackage{textcomp}
\usepackage{fix-cm}
\usepackage{array}
\usepackage{enumerate}
\usepackage{color}
\usepackage{comment}
\usepackage{subfigure}
\usepackage[colorinlistoftodos]{todonotes}
\usepackage[bookmarks]{hyperref}
\numberwithin{equation}{section}

\newtheorem{theorem}[equation]{Theorem}
\newtheorem{proposition}[equation]{Proposition}
\newtheorem{lemma}[equation]{Lemma}
\newtheorem{corollary}[equation]{Corollary}
\newtheorem{definition}[equation]{Definition}

\newtheorem{remark}[equation]{Remark}

\newcommand{\eqn}[1]{\label{eq:{#1}}}
\newcommand{\re}[1]{(\ref{eq:{#1}})}
\newcommand{\ct}[1]{\cite{kn:{#1}}}
\newcommand{\by}[1]{\bibitem{kn:{#1}}}
\newcommand{\be}{\begin{equation}}
\newcommand{\ee}{\end{equation}}
\newcommand{\ba}{\begin{array}}
\newcommand{\ea}{\end{array}}
\newcommand{\bea}{\begin{eqnarray}}
\newcommand{\eea}{\end{eqnarray}}

\newcommand{\nablu}{\nabla_{u}}

\newcommand{\uu}{{\bf u}}
\newcommand{\vv}{{\bf v}}
\newcommand{\rr}{{\bf r}}

\newcommand{\KK}{{\cal K}}
\newcommand{\RR}{{\mathbb R}}
\newcommand{\SD}{{\cal S}_\sigma}

\newcommand{\Rho}{\varrho}

\newcommand{\dirtht}{{\partial_t\vartheta\cdot\nabla_\vartheta}}
\newcommand{\dirthx}{{\partial_x\vartheta\cdot\nabla_\vartheta}}

\title{ Some remarks on  structured Keyfitz-Kranzer systems}
\author{Ralph Saxton\\Katarzyna Saxton}

\date{}
\begin{document}
\maketitle

\begin{abstract}
Several  applications for  systems of conservation laws of the form $U_t  + (\Phi (U) U)_x =0$, $U: R_t\times R_x\rightarrow R^n$ , $n\geq 2$, with $\Phi (U) = \phi (r, \Theta): R^n\rightarrow R$, $r = |U|$, and $\Theta = U/|U|\in S^{n-1}$, are obtained by imposing  structural conditions to provide a  classification framework for solutions dependent on the form of  $\phi(U)$. By prescribing the evolution of a particular eigenvalue, we can
categorize    classes of such functions, $\phi$, for which this evolution is met, into depending  on either a scalar field $z=rK(\Theta)$, where $K: S^{n-1}\rightarrow R$, or  on the vector field $\Theta$, and find for which the amplitude of solutions may blow up in finite time.   As a consequence, solutions to the corresponding Riemann problems can be divided into those with are classical and those which involve delta-shocks and/or vacuum states.

\end{abstract}

\section{Introduction}

The $n$-dimensional Keyfitz - Kranzer equations (\ct{KK}, \ct{KK2})  take the form
\be 
{\uu}_{t} + (\phi (\uu) \uu)_x={\bf 0},  
\eqn{KK}
\ee
where 
 $\uu: (t, x)\in{\mathbb R}_+\times{\mathbb R}\rightarrow\mathbb R^n$ and 
 $\phi :\mathbb R^{n}\rightarrow\mathbb R$. 
 On setting 
 \be
 {\cal  A}(\uu)=\nabla_{u}(\phi (\uu)\uu)=\uu\otimes\nabla_{u}\phi(\uu) + \phi(\uu){\bf  
I},
\eqn{A}
\ee
the  
characteristic polynomial for ${\cal A}(\uu)$ is given by
\begin{eqnarray} |\lambda{\bf I}-{\cal A}(\uu)|&=&|(\lambda  
-\phi(\uu)){\bf I}- \uu\otimes\nabla_{u}\phi(\uu)|\nonumber\\
&=&(\lambda -\phi(\uu))^{n}-(\lambda  
-\phi(\uu))^{n-1}tr((\uu\otimes\nabla_{u}\phi(\uu))\nonumber\\
&=&(\lambda -\phi(\uu))^{n-1}(\lambda - \phi (\uu)-\uu  
\cdot\nabla_{u}\phi (\uu)),  
\eqn{char}
\end{eqnarray}     
and its eigenvalues   are 
 \be
\lambda_{i}=\left\{\begin{array}{ll}
\phi(\uu)\equiv\lambda,& 1\leq i\leq n-1,\\
\phi(\uu)+\uu\cdot\nabla_{u}\phi(\uu)\equiv\mu,&i=n.  
\end{array}
\right. 
\eqn{eival}
\ee

Though at least two eigenvalues are equal for $n>2$, and the system is  then never strictly hyperbolic, it is simpler to use this terminology in the case that all eigenvalues are equal ($\uu\cdot\nabla_{u}\phi(\uu)=0$) for all $n\geq 2$, which defines $\Sigma$ below.
The corresponding first $n-1$ right eigenvectors of ${\cal A}(\uu)$
now satisfy the relations 
\be 
(\phi(\uu){\bf I}-{\cal A}(\uu))\rr_{i} =  
-\uu\otimes\nabla_{u}\phi(\uu)\rr_{i} =-\uu(\nabla_{u}\phi\cdot\rr_{i}),   1\leq i\leq n-1,
\eqn{n-1}
\ee
while the remaining eigenvector, $\rr_{n},$ satisfies
\be \begin{array}{ll}
((\phi(\uu)+\uu\cdot\nabla_{u}\phi  
(\uu)){\bf I}-{\cal A}(\uu))\rr _{n}& =
(\uu\cdot\nabla_{u}\phi (\uu){\bf  
I} -\uu\otimes \nabla_{u}\phi (\uu))\rr_{n} \\ 
&=(\uu\cdot\nabla  
_{u}\phi)\rr_{n}-\uu (\nabla_{u}\phi\cdot\rr_{n}).
\end{array}
\eqn{n}
\ee

 For $1\leq i\leq n-1,$ the $\rr_{i}$'s  can  be chosen  
as a set of linearly independent vectors belonging to the $n-1$ dimensional tangent space 
$\nablu\phi^{\perp}\equiv\{\vv\in\RR^n, \nabla_{u}\phi (\uu)\cdot\vv=0\}$ by \re{n-1},
while from \re{n} $\rr_{n}$ must be  
proportional to $\uu$, and so ${\cal A}(\uu)$ has a complete set of eigenvectors,
unless possibly $\uu\cdot\nabla _{u}\phi=0$. 

Let us define $\Sigma$ as the set
 \be\Sigma =  
\{{\bf u}\in \mathbb R^{n}: \uu\cdot\nablu\phi=0\},
\eqn{Sigma}
\ee 
on which $\lambda=\mu$. We then have two cases, one for $\nablu\phi|_\Sigma={\bf 0}$, which again leads to a complete set of eigenvectors, and the other for $\nablu\phi|_\Sigma\neq{\bf 0}$, which produces a right eigenvector deficiency since one must then have ${\bf r}_n\in\nabla_u\phi^\perp$,  by \re{n}.

In the same way, the first $n-1$ left eigenvectors, ${\bf  
l}_{i}$, belong to  $\uu^{\perp}\equiv\{\vv\in\RR^n, \uu\cdot\vv=0\}$, and  ${\bf l}_{n}$ is  either
proportional to $\nablu\phi$ unless $\uu\cdot\nabla _{u}\phi=0$, or ${\bf l}_{n}\in\uu^{\perp}$ if both $\uu  
.\nabla _{u}\phi=0$ and $\nabla_{u}\phi\neq{\bf 0}.$ In particular, a left eigenvector deficiency  occurs  on $\Sigma$ for $\nablu\phi|_\Sigma\neq{\bf 0}$.

Noting that the   characteristic fields $\lambda_i=\lambda$, $1\leq i\leq n-1$, satisfy  
$\rr_{i}\cdot\nablu\lambda=\rr_{i}\cdot\nablu\phi=0$,
 they are  everywhere linearly degenerate, (\ct {KK}). The $n^{\rm th}$  
characteristic field, $\lambda_n=\mu$, satisfies  
$\rr_{n}\cdot\nablu\mu\propto\uu\cdot\nablu(\phi+\uu\cdot\nablu 
\phi)$, except where $ \uu\cdot\nablu\phi =0$, which  reduces the right side  to the form ${\bf u}^T\cdot D_u^2(\phi)\cdot{\bf u}$.
As a result, the eigenvalue $\mu$ is  linearly degenerate on the set 
\be
\Upsilon= \{{\bf u}\in \mathbb R^n\setminus\Sigma:  
\uu\cdot\nablu(\phi+\uu\cdot\nablu\phi)=0\}\cup\{{\bf u}\in \Sigma:  
{\bf u}^T\cdot D_u^2(\phi)\cdot{\bf u}=0\}.
\eqn{Upsilon}
\ee

  Setting $\uu\equiv r\Theta, r>0, \Theta\in{\mathbb S}^{n-1}$, the eigenvalue coincidence set $\Sigma$, on which  $\lambda=\mu$, is denoted by
 $$\Sigma=\{(r, \Theta)\in{\mathbb R_+}\times{\mathbb S^{n-1}}: r\phi_r=0\},$$
 and the set $\Upsilon$, where $\mu$ becomes linearly degenerate, by
\be \Upsilon= \{(r, \Theta)\in{\mathbb R_+}\times{\mathbb S^{n-1}}\setminus\Sigma:  
r(r\phi)_{rr}=0\}\cup\{(r, \Theta)\in{\mathbb R_+}\times{\mathbb S^{n-1}}\in\Sigma:  
r^2\phi_{rr}=0\}.
\eqn{rupsilon}
\ee

We note that in results  motivated by elasticity (Keyfitz and Kranzer \ct{KK}, Liu and Wang \ct{LW}, Temple {\cite{Temple1983}) it is assumed that $\Upsilon=\{\emptyset\}$, $n=2$ and $\phi(r, \theta)\rightarrow \infty$ as $r\rightarrow 0+$ or $r\rightarrow +\infty$. 
With $\phi$ convex in $r$, $\Sigma$ then becomes a simple closed curve about $r=0$, and an eigenvector deficiency occurs when $\phi_\theta\neq 0$. \\

In this paper, we consider the consequences of $\mu$ being invariant along  $\mu-$characteristics. This restricts the functional dependence of $\phi(\uu)$ to two particular forms, ${\phi}_1(rK(\Theta))$, with $\phi_1:{\mathbb R}\rightarrow{\mathbb R}$  and $K:{\mathbb S}^{n-1}\rightarrow{\mathbb R},$ or ${\phi}_2(\Theta)$, with ${\phi}_2:{\mathbb S}^{n-1}\rightarrow{\mathbb R}$. This is the subject of the next section, after which we examine breakdown of solutions from smooth initial data, then the structure of formally self-similar solutions to the Riemann problem, which finally motivates a deeper examination of solutions with delta-shocks. Various applications are examined including models from chromatography and pressureless gas dynamics.

\section{Basic Results and Applications}

 It is  convenient to set $\uu = r\Theta$, where $r=|\uu|$, 
$\Theta \in S^{n-1}$. Writing the Keyfitz -  Kranzer system, \re{KK}, as 
 \be
 (r\Theta)_t+(\phi r\Theta)_x={\bf 0},
 \eqn{KKp}
 \ee
  gives
 \be
 \Theta(r_t+(\phi r)_x)+r(\Theta_t+\phi\Theta_x)={\bf 0}.
 \eqn{KKab}
 \ee
 Formally using the fact that $|\Theta|=1$, it is clear that for smooth $\Theta$ and $r>0$, 
 \be
 r_t+(\phi r)_x = 0
 \eqn{KKa}
 \ee
 and
 \be
 \Theta_t+\phi\Theta_x={\bf 0}.
 \eqn{KKb}
 \ee
 As a result, $\phi$ and $\Theta$ are  invariants, remaining constant on  the $\mu$, or $\lambda$, characteristics, respectively, since
\be 
D_\mu \phi \equiv (\partial_t+\mu\partial_x)\phi=0, \quad D_\lambda \Theta \equiv (\partial_t+ \lambda\partial_x) \Theta= {\bf 0}.
\eqn{inv}
\ee
 Equation \re{inv}i) follows  from  \re{eival} since $$\phi_t=\nabla_u\phi\cdot\uu_t=-\nabla_u\phi\cdot(\phi\uu)_x=-\phi_x\nabla_u\phi\cdot\uu -\phi\uu_x\cdot\nabla_u\phi,$$
 or
 $$ \phi_t+(\phi+r\phi_r)\phi_x=0,$$
 where $\mu=\phi+r\phi_r$ gives
 \be 
 D_\mu\lambda=0,
 \eqn{Riemann}
 \ee
 while equation \re{inv}ii)  comes directly from \re{eival} and \re{KKb}, so 
 \footnote{  We  note that  from  \re{Riemann2} and \re{polar} below,    
 \be
 D_\lambda \vartheta_i = (\partial_t+ \lambda\partial_x) \vartheta_i = 0, 1\leq i\leq n-1, \mbox{or}\, D_\lambda\vartheta=0.
 \eqn{Rpolar}
 \ee }
  \be 
  D_\lambda\Theta={\bf 0}.
  \eqn{Riemann2}
  \ee

The unit vector field $\Theta (t, x)$ can be written in terms of   polar  
coordinates 
$(\vartheta_1,\dots ,\vartheta_{n-1})\in\RR^{n-1}$, 
  which we   denote simply by $\vartheta (t, x)$,  with  $\vartheta_i\in [0, \pi], 1\leq i\leq n-2$, and $\vartheta_{n-1}\in[0, 2\pi)$, and for $\Theta=\Sigma_{i=1}^n\Theta_i{\bf e}_i$ in the standard euclidean basis, 
\be
 \Theta_{1}\equiv{\rm cos }\vartheta_{1},  \quad
\Theta_{j}\equiv\prod_{k=1}^{j-1}{\rm sin}\vartheta_{k}{\rm cos}\vartheta_{j} \,\,(2\leq j\leq n-1), \quad
 \Theta_{n}\equiv\prod_{k=1}^{n-1}{\rm sin}\vartheta_{k}.
 \eqn{polar}
 \ee 
 Where helpful, we may   choose between    representations of composed functions such as ${\cal F}=F\circ\Theta$, for  ${F}:  S^{n-1}\rightarrow\RR$, and   ${\cal F}: \RR^{n-1}\rightarrow\RR$. In such cases,  properties involving periodicity of ${\cal F}(\vartheta)$ will follow from those of $F ({\bf \Theta(\vartheta)})$ and \re{polar}. 

 
 Next we observe  a further invariant.
\begin{lemma}Let $r>0$. Then
\be D_{\lambda}(r^{-1}|\Theta_x|)=0=D_\lambda (r^{-1}|\vartheta_x|).
\eqn{link}
\ee
\end{lemma}
\begin{proof}  Combining $ D_\lambda \Theta_x+\phi_x\Theta_x={\mathbf 0}$ with  $D_\lambda r+r\phi_x=0$ gives
$$\Theta_{xt}+\phi\Theta_{xx}=[(\ln r)_t+\phi (\ln r)_x]\Theta_x$$
 which, taking the inner product of both sides with $\Theta_x$, leads to 
 $(\ln (r^{-1}|\Theta_x|)_t+\phi\, (\ln (r^{-1}|\Theta_x|)_x=0$. A similar calculation using \re{Rpolar} completes the proof.
 \end{proof}

 We  proceed using the following {\em structural} assumption.

\begin{definition}
 A constitutive hypothesis governing  $\phi$  is called {\em reduced} if this leads to classical solutions ${\bf u}(t, x)$  satisfying $D_\mu\mu=0$ for arbitrary, smooth, initial data ${\bf u}_0(x)$.  
 \eqn{reduced}
 \end{definition}
  \begin{remark}
 Although  $\lambda$ satisfies $ D_\mu\lambda=0$ identically, $\mu=(r\phi)_r$ is generally not   constant along $\mu$-characteristics.   Since $D_\mu (\mu-\lambda)=0$, however, in reduced systems  $\mu$-characteristics cannot  enter $\Sigma$ if $\mu-\lambda$ is initially nonzero.
 \end{remark}

 There exist nontrivial constitutive conditions  which lead to $\phi(\uu)$ being reduced.
  Denoting by $\partial_x\vartheta\cdot\nabla_\vartheta$  the operator $\Sigma_{i=1}^{n-1}\partial_x\vartheta_i\,\partial_{\vartheta_i}$, we have
 
\begin{lemma}
Let $\uu(t, x)$ be a smooth solution to equation \re{KK}. Then for $r>0$,
\be
D_\mu\mu=r\left\vert \begin{array}{ll}
\lambda_r & \partial_x\vartheta\cdot\nabla_\vartheta\lambda \\
\mu_r & \partial_x\vartheta\cdot\nabla_\vartheta\mu
\end{array}\right\vert,
\eqn{Dmu}
\ee
where the notation on the right stands for
\be
r\Sigma_{i=1}^{n-1}\;\partial_x\vartheta_{i}\cdot\left\vert \begin{array}{ll}
\lambda_r & \lambda_{\vartheta_i} \\
\mu_r & \mu_{\vartheta_i}
\end{array}\right\vert .
\ee
\end{lemma}
\begin{proof}
Recall from \re{Riemann}, $D_\mu\lambda=0$, where $\mu=\phi+r\phi_r$ and $\lambda=\phi$, thus $D_\mu\mu=D_\mu (r\phi_r)=(r\phi_r)_t+(\phi+r\phi_r)(r\phi_r)_x$. Together with equations \re{KKa} and \re{KKb}, we find 
$$
D_\mu(r\phi_r)=r_t\phi_r+r(\phi_r r_t+ \dirtht\phi_r)+\mu r_x\phi_r+\mu r(\phi_{rr}r_x+\dirthx\phi_r).
$$
But by equations \re{KKa} and \re{Rpolar}, this becomes 
$$
-(\phi r)_x\phi_r+r(-\phi_r (\phi_r r)_x-\phi\,\dirthx\phi_r)+\mu r_x\phi_r+\mu r(\phi_{rr}r_x+\dirthx\phi_r)
$$
which simplifies to
\be
r\partial_x\vartheta\cdot( \phi_r\nabla_\vartheta(r\phi_r)-(r\phi_r)_r\nabla_\vartheta\phi ),
\eqn{comm1}
\ee
or
\be
r\partial_x\vartheta\cdot( \phi_r\nabla_\vartheta(\phi+r\phi_r)-(\phi+ r\phi_r)_r\nabla_\vartheta\phi ),
\eqn{comm2}
\ee
thus establishing the result.
\end{proof}

\begin{lemma}\label{satz}
Let $\uu(t, x)$ be a smooth solution to equation \re{KK} subject to the ansatz 
$$\left\vert \begin{array}{ll}
\lambda_r & \partial_x\vartheta\cdot\nabla_\vartheta\lambda \\
\mu_r & \partial_x\vartheta\cdot\nabla_\vartheta\mu
\end{array}\right\vert=0.
$$
Then for $r>0$ this admits  three (potentially overlapping)  reduced forms for $\lambda=\phi(r, \vartheta)$ and $\mu$.
 Introducing a general function, ${\cal K}(\vartheta)$, these are 
 \begin{enumerate}
 \item $\lambda=\phi (r),\quad \mu=(r\phi(r))_r$,
 \item $\lambda=\phi (z), \quad \mu=(z\phi(z))_z,\quad z=r{\cal K}(\vartheta)$,
 \item $\lambda=\mu= \phi (\vartheta)$.
 \end{enumerate}

 The individual cases reduce $\Sigma$  to
   \begin{enumerate}
 \item $\{{\bf u}\in \mathbb R^{n}, r\phi_r(r)=0\}$,
 \item $\{{\bf u}\in \mathbb R^{n},\; z\phi_z(z)=0: i)\; {\cal K}(\vartheta)=0$, or ii) ${\phi}_z(z)=0\}$, 
 \item $\{\mbox{all }{\bf u}\in \mathbb R^{n}\}$,
 \end{enumerate}

and $\Upsilon$ becomes
    \begin{enumerate}
 \item $\{{\bf u}\in \mathbb R^{n}, r(r\phi(r))_{rr}=0\}$,
 \item $\{{\bf u}\in \mathbb R^{n},\; z(z\phi(z))_{zz}=0: i)\; {\cal K}(\vartheta)=0$, or ii) $(z\phi(z))_{zz}=0\}$, 
 \item $\{\mbox{all }{\bf u}\in \mathbb R^{n}\}$.
 \end{enumerate}
 \eqn{ansatz}
 \end{lemma}
 
 \begin{proof}
 When  dependence in $\phi(r, \vartheta)$ is reduced to either  $\phi(r)$, or to $\phi(\vartheta)$,    \re{Dmu} immediately shows $D_\mu\mu=0$, in which case the eigenvalues take the forms stated for  Case $1.$, or for Case $3.$, with  corresponding conclusions  following   for the sets $\Sigma$ and $\Upsilon$.
 
 We obtain the remaining case $2.$, assuming  $r\phi_r$ and $ \nabla_\vartheta\phi$ to be nonzero,
 which means the expression in \re{comm1} must be zero for all $r>0$ and all $\partial_x\vartheta$. Hence
 \be
 \phi_r\nabla_\vartheta(r\phi_r)\equiv(r\phi_r)_r\nabla_\vartheta\phi.
 \eqn{iden1}
 \ee
 This implies that $\partial_r((r\phi_r)^{-1}\nabla_\vartheta\phi)={\bf 0}$ and so
  \be
 \nabla_\vartheta\phi=r\phi_r{\bf v}(\vartheta)
 \eqn{a}
 \ee
where  the vector-field ${\bf v}(\vartheta)$ is
 arbitrary.  Moreover, ${\bf v}(\vartheta)$ is seen to be a gradient,  
 \be \nabla_\vartheta\phi=r\phi_r\nabla_\vartheta h(\vartheta),
 \eqn{h}
 \ee
  for some scalar field $h(\vartheta)$, since the components, \{$v_i(\vartheta), 1\leq i\leq n-1$\} of ${\bf v}(\vartheta)$ satisfy $v_{i,j}=v_{j,i}, 1\leq j\leq n-1$, where $_{,j}=\frac{\partial}{\partial\vartheta_j}$, as  a result of using \re{a} and \re{iden1}, which give
  \be
  v_{i,j}-v_{j,i}=\frac{-\phi_{,i}(r\phi_r)_{,j}+\phi_{,j}(r\phi_r)_{,i}}{(r\phi_r)^2}=\frac{r(r\phi_r)_r}{(r\phi_r)^3}(-\phi_{,i}\phi_{,j}+\phi_{,j}\phi_{,i})=0,
\eqn{v}
  \ee
  then making use of Poincar\'{e}'s Lemma.

  To find the dependence of $\phi$   if satisfying \re{h}, we  denote differentiation along
  a parametrized
curve $\gamma: \{r=r(s), \vartheta=\vartheta(s), s\in\RR\}$
by $d_s$. For a reduced system,
 \begin{eqnarray}\nonumber
 d_s\phi=&\phi_rd_sr+d_s\vartheta\cdot\nabla_\vartheta\phi\\\nonumber
 =&\phi_r(d_sr+rd_s\vartheta\cdot\nabla_\vartheta h(\vartheta))\\\nonumber
 =&r\phi_r(r^{-1}d_sr+d_sh(\vartheta))\\
 =&r\phi_rd_s\ln (r|\KK(\vartheta)|),\,  \eqn{chain}
 \end{eqnarray} 
 where $e^{h(\vartheta)}\equiv |\KK(\vartheta)|$.
 Thus $\phi$ remains constant either where $\phi_r=0$, as earlier, or on curves of constant $z=r\KK(\vartheta)$, with $\KK(\vartheta)\equiv K(\Theta\circ\vartheta)$  an arbitrary scalar field, at present, as can be verified by substituting in \re{iden1}.

Finally, with $\phi=\phi(z)$, the set $\Sigma$ now appears where $r\phi_r=r\partial_r\phi(r\KK(\vartheta)) =r\KK(\vartheta)\phi'(r\KK(\vartheta))$ $=z\phi_z(z)=0$, with  similar calculations leading to  the  results $z(z\phi)_{zz}=0$  in $\Upsilon\setminus\Sigma$ and $z^2\phi_{zz}=0$ in $\Upsilon\cap\Sigma$.
 \end{proof}
 
 Although  Case 1. clearly belongs under Case 2.   (with $\KK(\vartheta)=1$ above), here we will mainly focus on  $\phi$ having nontrivial angular dependence in which  the presence of zeros of $\KK(\vartheta)$ turns out to be necessary  for unbounded solutions to occur. In Case 1. all solutions remain bounded.
 
 Examples depending constitutively upon particular forms of $z$  appear in many  applications, and significant to many of these applications,  $K(\Theta)$ is given by an inner product, $z=r\,{\bf a\cdot\Theta}$  analyzed originally in \ct{K} with ${\bf a}=(1/2, 0)$ and $\phi(z)=z$,   later in \ct{YZ1} with ${\bf a}=(a, b)$ and $\phi'(z)>0$, and ${\bf a}=(a, b, c)$ with wave interactions in \ct{WS}, and elsewhere. Applications to chromatography appear in \ct{Su} for $\phi'(z)<0$, with ${\bf a}=(1, 1), \phi(z)=1/(1+z)$,  and in \ct{CY1}  with ${\bf a}=(-1, 1), \phi(z)=(2+z)/(1+z)$, both for $u_1, u_2 \geq 0$.  Most of these problems lead to solutions involving $\delta$-shocks,  as in the work of  Shelkovich,  \ct{S}, and others, extending  the system to arbitrary dimensions, 
 
 $$\partial_t u_i + \partial_x (u_i f_i (\mu_1u_1+\dots \mu_nu_n)=0, \quad1\leq i\leq n.$$

 \noindent
where it is shown there are solutions of the type $u_i=\hat{u}_i+e_i\delta(\Gamma)$, and $\Gamma$ denotes a discontinuity surface $\{(t, x):S(t, x)=0\}$ on which $\mu_1e_1(t, x)+\dots \mu_ne_n(t, x)=~0$. An example  not  involving a dot product is given in  \ct{Shen}, \ct{WLZH}, for  $z=r\sqrt{\Theta_1\Theta_2}, \, \Theta_i\geq 0, \,i=1, 2$, with $\delta$-shocks possible where $\Theta_1\Theta_2=0$.

 Examples in Case 3., $\phi=\phi(\Theta)$, include \ct{YZ2} in which $n=2$ and $\phi(\Theta)=\tilde{\phi}({\Theta_2}/{\Theta_1})$,  for $\Theta_1>0.$ These authors also considered a geometric optics model due to Engquist and Runborg (\ct{ER1}) for $\phi(\Theta)=\Theta_1$. Generalizations of zero pressure gas dynamics systems also fall under the same category (\ct{Sheng}, \ct{Y1}, \ct{LY}),
 \begin{eqnarray}\nonumber
 \rho_t+(\phi(m/\rho)\rho)_x=0,\nonumber\\
 (m)_t+(\phi(m/\rho)m)_x=0,\nonumber
 \end{eqnarray}
 where $\rho$ and $m$ represent density and momentum, respectively, since the correspondence $(\rho, m)=(u_1, u_2)$ gives
  \begin{eqnarray}\nonumber
  {u_1}_t+(\phi({u_2}/{u_1})u_1)_x=0,\nonumber\\
  {u_2}_t+(\phi({u_2}/{u_1})u_2)_x=0,\nonumber
  \end{eqnarray}
with ${u_2}/{u_1}={\Theta_2}/{\Theta_1}$, and the euler equations are recovered when $\phi$ is the identity function (\cite{LeVeque2004}). Motivated by an early paper of Kraiko (\cite{Kraiko}), the corresponding 3x3 system for mass, momentum and energy conservation (with internal energy) is treated in \cite{NilssonShelkovich2011_ApplAnal_I} and \cite{NilssonRozanovaShelkovich2011_ApplAnal_II}.
Further results on general distribution solutions can be found, for example, in \cite{CLi}, \cite{D4e}, \cite{Sever2002_JDE}, \cite{Sev}.

In the remainder of this paper, we analyze the cases

$$
\phi=\phi(r),\qquad
\phi=\phi(rK(\Theta)),\qquad
\phi=\phi(\Theta),
$$
which provide a  form of classification framework for reduced $\phi$ (\cite{SV}). The first  of these has an interpretation (\ct{KK}) relating the distinguished entropy
 $
\Psi(r)=\int_0^r s\,\phi(s)\,ds
$
to the total energy (kinetic plus strain) of an isotropic elastic string. This provides a link between constitutive properties for the elasticity system and inequalities across discontinuities in the Keyfitz-Kranzer model. A similar relation exists for a string  made of a nonisotropic material, with $r$ being replaced by $rK(\Theta)$, while the final case has no elasticity analogue but is found for pressureless gases.

In each of these cases, the issue appears that singularities in the form of Dirac masses are not generally defined under nonlinearities (the well-known $\delta^2$ question). In particular, for the case $z=ra\cdot\Theta$, this forces such solutions to escape to infinity in directions orthogonal to $a$ (\ct{S}), but otherwise allowing unrestricted left and right states for Riemann problems. Here we will find that our generalization to the form $z=rK(\Theta)$,  imposes a restriction that  left and right states may need to lie on an appropriate  manifold containing the states. In the absence of zeros of $K(\Theta)$, including the \lq r\rq\, case, such singularities do not appear. We also observe singularities, including those found in  known cases, in the broader context of $\Phi$ dependence on $\Theta$.

\section{ Breakdown of Smooth Solutions}

On examining   the evolution of initially smooth solutions, $\uu (t, x)$, to \re{KK},   potentially forming singularities in finite time, (\ct{KK}, \ct{KK2}), we  first observe basic properties of solutions  which can be understood 
particularly simply when $\phi$ is reduced. Subsequently, we address solutions for $\uu\notin\Sigma\cup\Upsilon$, or which lie in $\Sigma\setminus\Upsilon$, or in $\Sigma\cap\Upsilon$.

\subsection{Basic Results for Cases 1. and 2.}

As a result of  \re{reduced}, $D_\mu\mu=0$, it follows,  for $\mu_0(\beta)\equiv\mu(z_0(\beta))\in C^1$, and $z_0(\beta)\equiv z(0, \beta)$, that if the condition  $\mu_0'(\beta)=\mu_{z}(z_0(\beta))z_0'(\beta)<0$  is satisfied for some $\beta\in{\mathbb R}$, there exists a  maximal time $t^*~<~\infty$ by which $\mu_x (t, x(t))$ blows up along a related characteristic curve $\{x(t): \frac{dx}{dt}=\mu (t, x(t)),  \, x(0)=\beta, t\geq 0\}$. This comes by differentiating and showing $D_\mu\mu_x^{-1}=1$ so that $\mu_x^{-1}(t, x(t))=(\mu'_0(\beta))^{-1}(1+\mu'_0(\beta)t)$, while $\mu$ remains constant, $\mu(t, x(t))=\mu_0(\beta)$. Note that for $\mu_{z}(z_0(\beta))=0$, initial data lies in $\Upsilon$, and  $\mu_x (t, x(t))\equiv 0, t>0$.

Similarly, by equation \re{Riemann}, $D_\mu\lambda=0$, and  given $\lambda_0(\beta)\equiv\lambda(0, \beta)\in C^1$,
 one finds $D_\mu\ln|\lambda_x|=-\mu_x=-\mu_xD_\mu\mu_x^{-1}=D_\mu\ln|\mu_x|$. Therefore from the above,
 $\lambda_x^{-1}(t, x(t))=\lambda'_{0}(\beta)^{-1}(1+\mu_0'(\beta)t)$, where $\lambda_0'(\beta)=\lambda_{z}(z_0(\beta))z_0'(\beta)$, and any blowup of $\lambda_x$ coincides with that of $\mu_x$ unless  $\lambda_{z}(z_0(\beta))=0$, in which case the data lies in $\Sigma$, and $\lambda_x(t, x(t))\equiv 0, t>0$.

Finally, we observe that $z=r{\cal K}(\vartheta)$  is itself  invariant.
That is,
\be
D_\mu z=0,
\eqn{rz}
\ee
since
 \begin{eqnarray}\nonumber
 z_t+\mu z_x=&z_t+(\phi(z)z)_x\\\nonumber
 =&(r{\cal K}(\vartheta)_t+(\phi(z)r{\cal K}(\vartheta))_x\\\nonumber
 =&(r_t+(\phi(z)r)_x){\cal K}(\vartheta)+r({\cal K}(\vartheta)_t+\phi(z){\cal K}(\vartheta)_x)=0,
 \eqn{test}
 \end{eqnarray} 
 by \re{KKa} and \re{KKb}. From this, differentiating \re{rz} gives us either $z_x(t, x(t))=0$ if $z'_{0}(\beta)=0$, otherwise $D_\mu z_x^{-1}=\mu'_0(\beta)$ , so that
 \be z_x^{-1}(t, x(t))=z'_{0}(\beta)^{-1}(1+\mu_0'(\beta)t).
 \eqn{z11}
 \ee 
 Note  that the possibility of blowup for $z_x$ permits data to lie in $\Sigma$, while data in $\Upsilon$ preserves $z_x$.
 
 We remark, incidentally,  that if $y=r{\cal J}(\Theta)$, then $(\eta, q)\equiv (y, \phi(z)y)$ forms an entropy-entropy flux pair,  {\em i.e.} $\eta_t+q_x=0$ for any smooth function $\cal J$ of $\Theta$,  and which reduces to \re{KKa} for ${\cal J}=1$, or to the conservative form of $D_\mu z=0$ for ${\cal J}={\cal K}$, in the calculations above.

\subsection{The Case  $\lambda\neq\mu\quad({\bf u}\in\RR^n\setminus\Sigma)$.}

In this case   $\phi$ dependence can only depend on  $z$ (or simply on $r$). Establishing
$\mu-\lambda=r\phi_r(z)=z\phi_z(z)\neq 0$, given $D_\mu (z\phi(z))=0,$  requires that $K(\Theta)\neq\bf{0}$ for $t>0.$ However, because the range of values for $\Theta$ propagates along $\lambda-$characteristics by \re{inv},  assuming that $\phi_z(z)$ and $K(\Theta)$ are uniformly bounded away from zero for all $x$ at $t=0$  ensures that classical solutions  satisfy $\lambda\neq\mu$, $t\geq 0$.
Further, given that  $D_\mu z=0$ and so $rK(\Theta)$ remains constant  (nonzero for nonzero initial data) along $\mu-$characteristics,   $r$ remains uniformly bounded above and below  in terms of initial data on any interval of existence for classical solutions. 

Next, by equation \re{link}, $D_{\lambda}(r^{-1}|\vartheta_x|)=0$ means $|\vartheta_x|$  also remains bounded. Thus, since $z_x=r_x{\cal{K(\vartheta)}}+r\nabla_\vartheta{\cal{K}(\vartheta)}\cdot\vartheta_x$, $z_x$ can blow up in finite time if and only if $r_x$ does. Finally, $r_x$ alone  blows up for $\mu_0'(\beta)<0$, by \re{z11}.

We therefore arrive at the following result, in which $\phi$ is reduced, {\it i.e.} subject to  condition \re{reduced}.

\begin{proposition} Let ${\bf u}(t, x)$ be a solution to equation \re{KK}  with initial data $\uu (0, x)=\uu_0(x)\in C^1(\RR, \RR^n\setminus\Sigma).$ Then  $C^1-$breakdown of solutions $\uu(t, x)$ is possible in finite time only through blowup of  $r_x$. Both $|\vartheta_x|$ and $r$ remain bounded on any interval of existence. 

\end{proposition}

We next turn to  loss of hyperbolicity.
\subsection{The Case  $\lambda=\mu\quad ({\bf u}\in\Sigma)$.} 

In this section, we will assume that there exists an initial value $x=\alpha$ for which $\lambda(r_0(\alpha), \vartheta_0(\alpha))=\mu(r_0(\alpha), \vartheta_0(\alpha))$, where $r_0(\alpha)=r(0, \alpha)$ and $\vartheta_0(\alpha)=\vartheta (0, \alpha)$.
Given that $\lambda$ and $\mu$ remain constant along  $\mu -$characteristics under the conditions of Lemma \re{ansatz}  then, if, in particular   $\phi_r(r_0(\alpha),  \vartheta(\alpha))=0$ so that $\lambda=\mu$ at $t=0$, it follows that  $\lambda(t)=\mu(t)$ for $ t>0$. We denote such a  characteristic, $\{x=x(t): \lambda(t, x(t))=\mu (t, x(t)),  \, x(0)=\alpha, t\geq 0\}$,  a $\sigma-$characteristic, or simply  $\sigma$
 when there is no confusion, and we may  denote the derivative $D_\mu$  on $\sigma$ interchangeably with $D_\lambda$, but note that generally $\partial_xD_\mu\neq\partial_xD_\lambda$  there.

For $r>0$, the conditions on initial data and $\phi (r, \vartheta)$ which produce   $\sigma$-characteristics  are then, by Lemma \re{ansatz},

\begin{enumerate}
\item $\phi=\phi(r):\; \phi_r(r_0(\alpha))=0, $
 \item $\phi=\phi(r\KK(\vartheta)):  i) \, \KK(\vartheta_0(\alpha))=0,  \mbox{ or }
 ii)\, \phi_z(r_0(\alpha)\, \KK(\vartheta_0(\alpha)))=0,$
 \item $\phi=\phi(\vartheta):\; \vartheta_0(\alpha) \mbox{ arbitrary}.$
\end{enumerate}

The evolution of solutions on $\sigma$-characteristics depends  not only on reduced solutions lying in $\Sigma$ but also on whether solutions simultaneously lie in $\Upsilon$. For initial data to lie in $\Sigma\cap\Upsilon$,  added conditions, if needed  for $r>0$, are (see \re{rupsilon})

\begin{enumerate}
\item $\phi=\phi(r):\; \phi_{rr}(r_0(\alpha))=0,  $
\item $\phi=\phi(r\KK(\vartheta)): i) \, \KK(\vartheta_0(\alpha))=0,  \mbox{ or }
ii) \,\phi_{zz}(r_0(\alpha)\,\KK(\vartheta_0(\alpha)))=0, $
\item $\phi=\phi(\vartheta):\; \vartheta_0(\alpha) \mbox{ arbitrary.} $
\end{enumerate}

\subsubsection{Subcases 1. and 2.ii)}
Provided $r>0$, the radially dependent form of $\phi$ described by Subcase 1. can be regarded as a particular case of 2.ii) for $\KK= 1$ with $z=r$,  so we will here simply consider 2.ii), observing a few differences from 1.

Recall \re{rz}, 
\be
D_\mu z=0, \quad (D_\mu r=0),
\eqn{z} 
\ee
and  \re{z11},
\be
D_\mu z_x^{-1}=\mu'_0(\alpha)
\eqn{zx}
\ee
(or the equivalent equation in $r_x$).
Now, since $z_0=r_0\KK(\vartheta_0)$, we  obtain

\begin{lemma}{Let ${\bf u}(t, x)$ be a solution to equation \re{KK}  having initial data $\uu (0, x)=\uu_0(x)\in C^1(\RR, \RR^n)$ with $\phi_z(z_0(\alpha)))=0$ and $\KK(\vartheta_0(\alpha))\neq 0$, so that $\uu_0(\alpha)\in\Sigma$ generates a corresponding $\sigma-$characteristic. On $\sigma$, breakdown of solutions $\uu(t, x)$ is possible in finite time only through blowup of $r_x$, while $\phi_x\equiv 0$. Both  $|\vartheta_x|$ and $r$ remain constant along $\sigma$ on any interval of existence. }

\end{lemma}
\proof
Since for $\uu_0(\alpha)\in\Sigma$ we  have $\lambda=\mu$ along the  $\sigma-$characteristic, therefore $z\phi_z(z)=0$,  $z=r\KK(\vartheta)$. Because $\vartheta$ remains constant on $\sigma$, then since $\KK(\vartheta_0(\alpha))\neq 0$ it follows that $\phi_z(z)=0$ on $\sigma$. In particular $\phi_x=\phi_zz_x=0$  (also $\phi_r=\phi_z\KK(\vartheta)=0$). Further, since $D_\lambda r=-r\phi_x=-r\phi_zz_x=0$, then $r$ also remains constant, which implies in turn that $|\vartheta_x|$  remains constant on  $\sigma$, by \re{link}.

Finally, \re{z11} shows finite time blowup of $z_x$ occurs by    $t^*=-\mu_0'(\alpha)^{-1}>0$ provided that $z_0'(\alpha)\neq 0$ and $\mu_0'(\alpha)<0$, and since all terms in $z_x$ remain constant on $\sigma$ save for $r_x$, the result follows.
\qed 

Similarly,
\begin{corollary}
If $\uu_0(\alpha)\in \Sigma \cap \Upsilon$ 
with $\KK(\vartheta_0(\alpha))\neq 0$, we  have $\phi_{zz}(z_0)~=~0\, \newline (\mu_0'(\alpha)=0)$, and in this case $r_x$, as for $z_x$, remains constant on $\sigma$.
\end{corollary}

\subsubsection{Subcase 2 i).}

Here we have $\uu_0(\alpha)\in\Sigma\cap\Upsilon$, due to
  ${\cal K}(\vartheta_0(\alpha))=0$. Equations  \re{z} and \re{zx} from the previous section again apply.

\begin{lemma}{Let ${\bf u}(t, x)$ be a solution to equation \re{KK}  having initial data $\uu (0, x)=\uu_0(x)\in C^1(\RR, \RR^n)$ with  $\KK(\vartheta_0(\alpha))= 0$ $(z_0(\alpha)=0)$, such that $\uu_0(\alpha)\in\Sigma$ generates a corresponding $\sigma-$characteristic. On $\sigma$, breakdown of solutions $\uu(t, x)$ is possible in finite time  through blowup of $r$ and $|\vartheta_x|$ provided $\nabla_\vartheta\KK(\vartheta_0(\alpha))\neq{\bf 0}$. }
\end{lemma}

\proof 
Since $D_\mu z=0$, $\KK(\vartheta_0(\alpha))= 0$,  then  $\phi_z(z)=\phi_z(0)$ on $\sigma$ 
and $\mu_0'(\alpha)=\mu_z(z_0(\alpha))z_0'(\alpha)=\mu_z(0)z'_0(\alpha)=2\phi_z(0)z'_0(\alpha)$,
and for $\phi_x(0)\neq 0$, equation \re{zx} now reduces on $\sigma$ to 
\be
D_\lambda z_x^{-1}=2\phi_{z}(0)z_0',
\eqn{nontrivialblow}
\ee
from which finite time blowup of $z_x$ occurs by  time $t^*=-(2\phi_{z}(0)z_0')^{-1}>0$ provided that $\phi_{z}(0)z_0'<0$. On $\sigma$, $z_x$ reduces to $r\nabla_\vartheta \KK(\vartheta_0(\alpha))\cdot\vartheta_x$ since $\KK(\vartheta_0(\alpha))=0$. As a result,    it holds that
\be
|z_x|=r|\nabla_\vartheta\KK(\vartheta_0)\cdot\vartheta_x|\leq r_0|\vartheta'_0 |^{-1}|\nabla_\vartheta \KK(\vartheta_0(\alpha))||{\vartheta_x}|^2
\eqn{zxbelow}
\ee
due to  \re{link}, and so $|\vartheta_x|$, and therefore  $r$, again by  \re{link},  blow up by $t=t^*$, provided 
$\nabla_\vartheta\KK(\vartheta_0)\neq{\bf 0}$.  \qed
\begin{remark}
Under the same conditions, on $\sigma$, equation \re{KK} implies $$D_\lambda\uu=-\phi_z(0)z_x\uu$$
and, as result we have the direct formula $$\uu (t, \alpha+\phi(0)t)=(1-t/t^*)^{-1/2}\,\uu_0(\alpha).$$
\end{remark}
\begin{remark}
If $\uu_0(\alpha)\in\Sigma\cap\Upsilon$, having
  ${\cal K}(\vartheta_0(\alpha))=0$ and $\nabla_\vartheta\KK(\vartheta_0(\alpha))={\bf 0}$, then $z_x=0$ and $\vartheta_x=\vartheta_0'(\alpha)$ on $\sigma$.
\end{remark}
\proof That $z_x=0$ is an immediate consequence of \re{zxbelow}, while $\vartheta_x$ remaining constant on $\sigma$ comes by using this result  and differentiating \re{Rpolar} in $x$ to find that $D_\lambda\vartheta_x=0$.\qed

\subsubsection{Subcase 3.}
\label{theta section}
\noindent In the remaining case $\phi=\phi(\vartheta)$, $\uu\in\Sigma\cap\Upsilon$ for arbitrary $\uu_0$ since $\phi_r=0$.
 So $D_\mu\mu=0$  becomes $D_\phi\phi=0$, which leads to $D_\phi (\phi_x)^{-1}=1$, from which $\phi_x\equiv \nabla_\vartheta\phi\cdot\vartheta_x\rightarrow-\infty$ in finite time $t_*=-(\nabla_\vartheta\phi\cdot\vartheta'_{0})^{-1}>0$ provided $\nabla_\vartheta\phi\cdot\vartheta'_{0}<0$. Using  equation \re{link} again, $r=r_o|\vartheta_x|/|\vartheta'_0|$, then shows that $|\vartheta_x|$, $r\rightarrow+\infty$ as $t\rightarrow t_*^-$. \qed
\begin{remark}
Under these  conditions, on $\sigma$,  equation \re{KK} here  implies $$D_\lambda\uu=-\phi_x\uu$$ where 
$\phi_x=\nabla_\vartheta\phi\cdot\vartheta'_{0}\,(1+t\nabla_\vartheta\phi\cdot\vartheta'_{0})^{-1}$, from which we obtain the local formula 
$$\uu (t, \alpha+\phi(\vartheta_0(\alpha))\,\,t)=(1+\nabla_{\vartheta{_0}}\phi(\vartheta_0(\alpha))\cdot\vartheta'_{0}(\alpha)t)^{-1}\,\uu_0(\alpha).$$
\end{remark}
\noindent See sections \ref{deltatheta} and \ref{gas} for some examples.
 \begin{remark}  
The observations above, that $r(t, x)$ may become unbounded in finite time  in cases 2i) and 3., coincides with  their having right eigenvector deficiencies.  In both cases, $\uu.\nabla_u\phi|_\Sigma=0.$ If  $\nablu\phi|_\Sigma\neq{\bf 0}$ also,  this comes from \re{n}, which then implies that  ${\bf r}_n\in\nabla_u\phi^\perp$, from which the result  follows.
\end{remark}
\noindent To see this is a consequence in case 2., we first calcuate
\be
\eqn{def1}
\nabla_u\phi=\phi'(rK(\Theta))\nabla_u(rK(\Theta)),
\ee 
where $\nabla_u(rK(\Theta))=\nabla_ur\,K(\Theta)+r\nabla_\Theta K(\Theta)\cdot\nabla_u\Theta$. Given  $\Theta=r^{-1}{\bf u}$ implies  $\nabla_u r=\Theta$ and  $\nabla_u\Theta=r^{-1}({\bf I}-r^{-2}{\bf u}\otimes{\bf u})=r^{-1}({\bf I}- \Theta\otimes\Theta)$,   we have, retaining a final (zero) term, that
\be
\eqn{def2}
\nabla_u(rK(\Theta))=\Theta\,K(\Theta)+\nabla_\Theta\,K(\Theta) - \Theta \,(\Theta\cdot\nabla_\Theta\,K(\Theta)).
\ee
\noindent In case 3., similarly,
\be
\eqn{def3}
\nabla_u\phi(\Theta)=r^{-1}(\nabla_\Theta\,\phi(\Theta) - \Theta \,(\Theta\cdot\nabla_\Theta\,\phi(\Theta))).
\ee

From \re{def1} and \re{def2}, we  recover ${\bf u}\cdot\nabla_u(\phi (rK(\Theta)))=\phi'(rK(\Theta))\,rK(\Theta)$, which is zero on $\Sigma$ in both cases 2i) and 2ii). In case 3., ${\bf u}\cdot\nabla_u\phi(\Theta)$ is identically zero by \re{def3}  (and $\Sigma$ is $\RR^n$ in this case). However, since $K(\Theta)={\bf 0}$ on $\Sigma$ in case 2i), the terms on the right of equation \re{def2}  only reduce to zero on $\Sigma$ when  $\nabla_\Theta K(\Theta)$ is parallel to $\Theta$ - in other words when $K(\Theta)$ locally behaves as a function of $|\Theta|$, which is  constant because $\Theta\in S^{n-1}$, in which case $\phi(rK)$ reduces locally to a  function of $r$ alone. A similar argument holds for case 3.,  showing the expressions $\nabla_u(rK(\Theta))$ and $\nabla_u\phi(\Theta)$ are generally nonzero in   cases 2i) and 3., appropriately.

\subsubsection{Subcase 2i)  with Subcase 2ii).}

If $\KK(\vartheta)$ is continuous and bounded away from zero for all $\vartheta$, and if $c$ has the same sign as $\KK(\vartheta)$, then we may assume that the solution set for $r\KK(\vartheta)=c$ is defined by a closed curve or surface about the origin, with $0<r<\infty$. However, if there are directions for which $\KK=0$, 
this leads to the following possibility.
Setting $C=\{c\in\RR\setminus\{0\}: \phi_z(c)=0\}$ and $D=\{d\in\RR^{n-1}: \KK(d)=0\}$,
assume  that neither set is empty, in which case $\Sigma_{D}=\{(r, d)\in\RR\otimes\RR^{n-1}: \KK(d)=0\}$ 
is unbounded in $\RR^n$. Thus the set $\Sigma=\{(r, \vartheta)\in\RR\otimes\RR^{n-1}: r\KK(\vartheta)\in C \mbox{ or } \vartheta\in D\}\equiv\Sigma_C\cup\Sigma_D$, where $\Sigma_{C}=\{(r, \vartheta)\in\RR\otimes\RR^{n-1}: \,r\KK(\vartheta)=c\in\RR\setminus\{0\}, \,\phi_z(c)=0\}$, is likewise nonempty while $\Sigma_C\cap\Sigma_D=\emptyset$. In particular if $\KK(\vartheta_d)=0$, and $\KK(\vartheta)$ is continuous in a  neighbourhood of $\vartheta_d$  where $(r, \vartheta)$ satisfies  $r\KK(\vartheta)=c$, then $r\rightarrow\infty$ as $\vartheta\rightarrow\vartheta_d$.  In this case 
 $\Sigma_C$ is also unbounded in $\RR^n$ and the presence of any zeros in $\KK(\vartheta)$ rules out $\Sigma$ forming a closed curve/surface about $r=0$.

As a result of the above remarks when, for example,  $\uu_0(\alpha)$ has $z_0\in C$, and  satisfies
     $\phi_z(r_0(\alpha)\KK(\vartheta_0(\alpha)))=0$ for all $\alpha\in [\alpha_c , \alpha_d)$ and some $\alpha_c<\alpha_d$, and ${\cal K}(\theta_0(\alpha_d))=0$  
      so that ${\bf u}_0(\alpha)\in\Sigma$ for all $\alpha\in [\alpha_c, \alpha_d]$, then $|{\bf u}_0(\alpha)|$ must be unbounded on $[\alpha_c, \alpha_d]$. 
      This results in an unbounded solution with the part of the solution originating on $\Sigma_C$  remaining bounded, since both $r$ and $\vartheta$ initially remain approximately constant, though with  $r_x$  potentially leading to blow up, but with the part originating from $\Sigma_D$ having $r(\vartheta_d, t)$ unbounded.
     
 \section{Classical Riemann solutions}
 
In this,  and in the following section, we  consider  the standard Riemann problem, with initial data 
 \be
 \uu(x, 0)=\uu_-+H(x)(\uu_+-\uu_-),\eqn{riemann}
 \ee 
 where $H(x)$ is the unit step function and $\bf u_-$ and $\bf u_+$ denote constant left ($x<0$) and right ($x>0$) states.
 In particular,  we  examine self-similar solutions of the form $\uu(x, t)=\uu(\xi),\, \xi=x/t$, 
 (see \ct{YZ1}). The results of this section  connect to  the  analysis of evolving singularities which follows.
 
 \subsection{The Case $\phi=\phi(r)$}
The radial case, $K\equiv1$, differs from that of Case 2. when using the conditions on $K$ below, and is  known to possess global $L^\infty$ solutions for $L^\infty$ initial data, (\cite{Ch},  \cite{ACFS}). Obtaining classical solutions is standard.

\subsection{The Case $\phi=\phi(z)$}
 
 Throughout this section, we will make the assumption that $\phi$ satisfies the  following hypotheses (see Lemma \ref{satz}),
\be
\mbox{ i) } \phi(\uu)\equiv\phi(z)=\phi(r\KK(\vartheta)),  \mbox{ ii) }\phi(0)=0, \mbox{ iii) } \phi_z(z)>0 \mbox{ and } \mbox{ iv) } (\phi(z)z)_{zz}>0.
  \eqn{phihyp}
 \ee
 which implies that $ z\in\Sigma\cap\Upsilon\Leftrightarrow\KK(\vartheta)=0$ (see Case {2.i)}).
 We also assume that $\KK(\vartheta)$ satisfies
 
 \be
 \mbox{ v) }\quad\quad \KK(\pi-\vartheta_1, \cdots, \pi-\vartheta_{n-2}, \pi+\vartheta_{n-1})=-\KK(\vartheta_1, \cdots, \vartheta_{n-2}, \vartheta_{n-1}),\quad
 \eqn{odd}
 \ee
 which means that  $K(\Theta)$ is odd in $\Theta$ (recall $\KK(\vartheta)\!= \!K(\Theta\,\circ\,\vartheta)$)  and, in particular, given any null-element, $\Theta_\sigma$, of $K$, its antipode, $-\Theta_\sigma$, is another. We  adopt the non-degeneracy condition 
 \be
\KK(\vartheta)={0}\Rightarrow \nabla_\vartheta\KK(\vartheta)\neq{\bf 0}.  
 \eqn{nond}
 \ee
 We begin by considering the possibility of connecting  left and right states by a single shock, contact, or rarefaction wave.
 The classical Rankine-Hugoniot conditions for \re{KK} are
 \be
 s[\uu]=[\phi(\uu)\uu].
 \eqn{RH}
 \ee
 where $s$ gives the speed of a shock, and $[\uu]=\uu_--\uu_+$ denotes the jump of $\uu$ across the discontinuity.
 Substituting $\uu=r\Theta$ in \re{RH}  gives
 
 \be
 (s-\phi(z_+))r_+\Theta_+=(s-\phi(z_-))r_-\Theta_-.
 \eqn{RH1}
 \ee
 As a result, there are two cases to consider: contact discontinuities, $ J$,  in which case,  by \re{phihyp},
 \be
 s=\phi(z_+)=\phi(z_-)  \quad \makebox{ and so  \, } r_+\KK(\vartheta_+)=r_- \KK(\vartheta_-),
 \eqn{J}
 \ee
and jump discontinuities,  $\cal{S}$,  $\phi(z_+)\neq\phi(z_-)$, where  by \re{RH1} $\Theta_+=\pm\Theta_-$.
In either  case (see \re{odd}),
\be
s=\frac{\phi(z_-)z_--\phi(z_+)z_+}{z_--z_+}.
\eqn{srz}
\ee
 We note that  \re{srz} is the Rankine-Hugoniot relation corresponding to $D_\mu z=0$ (see \re{z}), {\em i.e.} 
\be
z_t+(\phi(z)z)_x=0,
\eqn{z1}
\ee

We  also note here that the ordering of the eigenvalues \re{eival} depends  on whether $z$ is positive or negative,
\be
\mbox{ for }z\geq 0: \quad \lambda(z)=\lambda_1=\dots=\lambda_{n-1}\leq\lambda_n=\mu(z),
\ee
while
\be
\mbox{ for }z\leq 0: \quad \mu(z)=\lambda_n\leq\lambda_{n-1}=\dots=\lambda_{1}=\lambda(z).
 \ee
 Due to the entropy conditions, (see \ct{Lax}, \ct{KK}, \ct{LW}), for $z_+<z_-<0$,
  \be
 \mu(z_+)<s<\mu(z_-),\quad s<\lambda(z_+),\quad s<\lambda(z_-),
 \eqn{ent1}
 \ee
two such states  can be connected by a backward shock ($\overleftarrow{S}$) 
 while, for $0<z_+<z_-$, the states can be connected by a forward shock ($\overrightarrow{S}$), for which
  \be
 \mu(z_+)<s<\mu(z_-),\quad s>\lambda(z_+),\quad s>\lambda(z_-).
 \eqn{ent2}
 \ee
In each of these cases,  Lax shocks exist for $\Theta_+=\Theta_-$, with $\, z_+z_->0$.
 In the case $\, z_+z_-<0$   (for $z_+<z_-$), which includes the possibility of $\Theta_+=-\Theta_-$, such a connection is overcompressive, 
  
 \be
 \mu(z_+)\leq\lambda(z_+)\leq s\leq\lambda(z_-)\leq\mu(z_-).
 \eqn{com}
 \ee

We next set $\uu(x, t)=\uu(\xi), \xi=x/t$ to examine self-similar solutions. Equation \re{KK} for Riemann data \re{riemann} then reduces to the system
 \be
 (-\xi+\phi)\uu_\xi+\phi_zz_\xi\uu=\bf{0}.
 \eqn{ssu}
 \ee
On substituting $\uu=r\Theta$, this becomes
\be
((-\xi+\mu)r_\xi+\phi_zr^2\KK_\xi)\Theta+(-\xi+\lambda)r\Theta_\xi={\bf 0}.
\eqn{5.6}
\ee

 Since $\Theta\in S^{n-1}$,  $\Theta_\xi$ is orthogonal to $\Theta$. If $\Theta_\xi={\bf 0}$ it follows that $\KK_\xi(\vartheta)=0$, and so, by examining the first coefficient of \re{5.6}, either $r_\xi=0$ or $\xi=\mu$. That is, for constant $\Theta$ either $\uu$ is constant or $\xi=\mu$.
 If $\Theta_\xi\neq{\bf 0}$, both coefficients in \re{5.6} must be zero.  This implies that $\xi=\lambda$, and the first coefficient then becomes

 $$\phi_z(r\KK r_\xi+r^2\KK_\xi)=r\phi_zz_\xi=r\phi_\xi,$$
which must also be zero. Thus, for $\Theta_\xi\neq {\bf 0}$,  $\phi(z(\xi))$ remains constant for $r>0$  as does $z=r\KK(\vartheta)$, since $\phi_z>0$ from \re{phihyp}.
 
 Consequently, if $\xi=\lambda$, we have a contact discontinuity, $ J$, 
 \be
 z=r\KK(\vartheta)=r_-\KK(\vartheta_-)
 \eqn{xilambda}
 \ee
  while if $\xi=\mu$, we have  a rarefaction, $R$,
  \be
  \Theta=\Theta_- \quad (\vartheta=\vartheta_-),
  \eqn{ximu}
  \ee
 otherwise $\uu(\xi)$ is constant. 
 In particular, two states, with $z_-<z_+<0$, can be connected by a backward rarefaction ($\overleftarrow{R}$) 
 while, for $0<z_-<z_+$, the states can be connected by a forward rarefaction ($\overrightarrow{R}$).
 
 By using the entropy conditions \re{ent1}, \re{ent2}, two general states $\uu_-$, $\uu_+$ can be connected by classical (compound shock, contact and/or rarefaction) waves under the following conditions (\ct{YZ1}),
 
 $$J+\overrightarrow{S}\, (z_->z_+>0), \quad J+\overrightarrow{R}\, (0\leq z_-<z_+),$$
  
  $$\overleftarrow{S}+J\,(0>z_->z_+), \quad \overleftarrow{R}+J\, (z_-<z_+\leq0),$$
  
  $$\overleftarrow{R}+\overrightarrow{R}\,(z_-<0<z_+).$$
 The remaining case, for which $z_-\geq 0\geq z_+$,  does not have  states $z_-$, $z_+$ connected by classical waves that are admissible through the entropy conditions. 
 Using  $\pm$ superscripts to denote conventional limits from above and below and retaining $\pm$ subscripts to indicate left and right states, 
 we  begin by examining  the limit case, $z_-\geq z_+> 0$ as $z_+\rightarrow 0^+$, of   $J+\overrightarrow{S}$ solutions
 \noindent $(\mbox{as }\vartheta_+\mbox{ approaches }\vartheta_\sigma^+ \mbox{\, with \,} \KK(\vartheta_\sigma)=0)$, 
  
 $$z_-\stackrel{J}{\longrightarrow} z_*\stackrel{\overrightarrow{S}}{\longrightarrow} z_+, \quad z_+\rightarrow 0^+$$
 where $z_-=r_-\KK(\vartheta_-)=z_*=r_*\KK(\vartheta_+)$, $\quad z_+=r_+\KK(\vartheta_+)$, to find that
 $$r_*=\frac{r_-\KK(\vartheta_-)}{\KK(\vartheta_+)}\stackrel{\vartheta_+\rightarrow\vartheta_\sigma^+}{\longrightarrow} +\infty$$
 as the forward shock and contact discontinuity coincide, while  
 $$
 s=\frac{\phi(z_*)z_*-\phi(z_+)z_+}{z_*-z_+}\rightarrow\phi(z_-)^+$$
 as $z_+\rightarrow 0^+$.
 
  Next, we examine  $z(\xi)$ and $\uu(\xi)$. In terms of $\xi=x/t$ again, we have from equations  \re{z1} and  \re{ssu},
  
  \be
  -\xi z_\xi + (\phi (z)z)_\xi=0,
  \eqn{zxi}
  \ee
 and
  \be
  -\xi \uu_\xi + (\phi (z)\uu)_\xi={\bf 0}.
  \eqn{rxi}
  \ee
 Integrating between $J$ and $\overrightarrow{S}$ as $z_+\rightarrow 0^+$, \re{zxi} then implies that
 \begin{eqnarray}
 0&=&\lim_{z_+\rightarrow 0^+}\int_{\phi(z_-)}^{s}(-\xi z_\xi + (\phi (z)z)_\xi)d\xi\nonumber\\
&& =\lim_{z_+\rightarrow 0^+}( (-\xi z +\phi (z)z)\vert_{\phi (z_-)}^s)+\int_{\phi(z_-)}^{\phi(z_-)^+}zd\xi\nonumber \\
&& =\int_{\phi(z_-)}^{\phi(z_-)^+}z d\xi. \eqn{d1}
 \end{eqnarray}
Similarly, by \re{rxi},
  \begin{eqnarray}
 {\bf 0}&=&\lim_{z_+\rightarrow 0^+}\int_{\phi(z_-)}^{s}(-\xi \uu_\xi + (\phi (z)\uu)_\xi)d\xi\nonumber\\
&& =\lim_{z_+\rightarrow 0^+}( (-\xi \uu +\phi (z)\uu)\vert_{\phi (z_-)}^s)+\int_{\phi(z_-)}^{\phi(z_-)^+}\uu d\xi\nonumber \\
&& = - \phi(z_-) \uu_+ +\int_{\phi(z_-)}^{\phi(z_-)^+}\uu d\xi. \eqn{d2}
 \end{eqnarray}

   The analogous result  for  the limit case $z_-\rightarrow 0^-$ $(\vartheta_-\rightarrow\vartheta_\sigma^-)$ for   $\overleftarrow{S}+J$ solutions shows that
 $$z_-\stackrel{\overleftarrow{S}}{\longrightarrow} z_*\stackrel{J}{\longrightarrow} z_+, \quad z_-\rightarrow 0^-$$
 where $z_-=r_-\KK(\vartheta_-),\quad z_*=r_*\KK(\vartheta_-)=z_+=r_+\KK(\vartheta_+)$, and the intermediate state $r_*$  again becomes unbounded, while
 
 $$
 s=\frac{\phi(z_-)z_--\phi(z_*)z_*}{z_--z_*}\rightarrow\phi(z_+)^-$$
 as $z_-\rightarrow 0^-$.
 
It follows, integrating \re{zxi} between $\overleftarrow{S}$ and $J$, that

 \begin{eqnarray}
 0&=&\lim_{z_-\rightarrow 0^-}\int^{\phi(z_+)}_{s}(-\xi z_\xi + (\phi (z)z)_\xi)d\xi\nonumber\\
&& =\int_{\phi(z_+)^-}^{\phi(z_+)}z d\xi, \eqn{d3}
 \end{eqnarray}
and from \re{rxi}, that
  \begin{eqnarray}
 {\bf 0}&=&\lim_{z_-\rightarrow 0^-}\int^{\phi(z_+)}_{s}(-\xi \uu_\xi + (\phi (z)\uu)_\xi)d\xi\nonumber\\
&& = - \phi(z_+) \uu_-+ \int_{\phi(z_+)^-}^{\phi(z_+)}\uu d\xi. \eqn{d4}
 \end{eqnarray}
 
 Equations \re{d2}, \re{d4} imply that $\uu(\xi)$ contains weighted delta singularities  in the respective limits for $\xi=\phi(z_-)$ or $\xi=\phi(z_+)$ whereas $z(\xi)$ remains of bounded variation by equations \re{d1}, \re{d3}.
 
 For the case $z_->0>z_+$, 
 integrating \re{zxi} from $s^-$ to $s^+$ leads to the following,

 \be
 \int_{s^-}^{s^+}zd\xi =\alpha_-\KK_-+\alpha_+\KK_+=0,\quad \alpha_\pm\geq 0,
 \eqn{zKK}
 \ee
  where $\KK_\pm=K(\Theta_\pm)$, while  integrating 
 \re{rxi} gives
 
\be
\int_{s^-}^{s^+}\uu d\xi = \alpha_-\Theta_-+\alpha_+\Theta_+.
\eqn{uKK}
\ee 

 As a result, we have
 \begin{lemma}
  For $z_-\geq 0 \geq z_+$, any self-similar solution, $\uu(\xi)$, to \re{rxi} connecting constant states $\uu_-, \uu_+$ is associated with a corresponding self-similar solution, $z(\xi)$, to \re{zxi}.
  If $\uu(\xi)$ contains a $\delta$-shock then $z(\xi)$ has a shock, both moving at the speed $s$ given by \re{srz}, and satisfying \re{zKK}, \re{uKK}, where $\alpha_{\pm}$ are given by
   
 \be \alpha_\pm=\pm (s-\phi(z_\pm))r_\pm\geq 0.  
 \eqn{oc}
 \ee
 \end{lemma}
 \proof   The inequality \re{oc} is  a consequence of the overcompressibility condition \re{com} and the hypotheses made in \re{phihyp}. 
 The expression $z(\xi)$ satisfies \re{zKK} (and \re{z1}), or more fully,
 \be \int_{s^-}^{s^+}zd\xi 
 = -s(z_--z_+)+z_-\phi(z_-)-z_+\phi (z_+)=0,
 \eqn{z-+}
 \ee
 by \re{srz}. Similary, from \re{uKK}, $\uu(\xi)$ satisfies

    \begin{eqnarray}
 &\int_{s^-}^{s^+}\uu d\xi &= (s-\phi (z_+))\uu_++(\phi (z_-)-s)\uu_-\nonumber\\
&&=\frac{\phi(z_-)-\phi(z_+)}{z_--z_+}r_-r_+(\KK_-\Theta_+-\KK_+\Theta_-),
 \eqn{uu-+2}
 \end{eqnarray} 
 where we note that \re{uu-+2} can only reduce to zero if $\Theta_+=\pm\Theta_-$. Otherwise,  any  solutions, ${\bf u}$, containing delta-shocks at $\Theta_\sigma, (K(\Theta_\sigma)=0)$, are associated with bounded jumps in $z$ across $z=0$.

 \qed

 It will be seen in section \ref{sec:zdelta} that in certain cases \re{uu-+2} is automatically satisfied by the simplest form of  solution containing a delta-shock for which $\Theta_\sigma$ is well-defined for general left and right states, however it is not always the case that such solutions exist except on a locus defined by these states.

\subsection{The Case $\phi=\phi(\Theta)$}\label{sec:thetanice}

Recall that here $\Sigma=\Upsilon=\mathbb R^{n}$, by Lemma \ref{satz}, and the consequences of having $\mu=\lambda$ will turn out generally to allow only contact discontinuities and \lq cavitation\rq\, states to exist in the absence of singular solutions.
 
 As previously, we first try to connect left and right states by a shock or contact discontinuity using \re{RH},
 \be
 (s-\phi_+)r_+\Theta_++(\phi_--s)r_-\Theta_-={\bf 0},\quad \phi_\pm=\phi(\Theta_\pm).
 \eqn{rhtheta}
 \ee 
 
\noindent Thus  for $r_{\pm}>0$ either $s=\phi_+=\phi_-$ (a contact discontinuity) or $\Theta_+=\pm\Theta_-$. In the  case that $\Theta_+=\Theta_-$,  so $(s-\phi_-)r_++(\phi_--s)r_-=0$ or  $(s-\phi_-)(r_+-r_-)=0$, either $r_+=r_-$ and then  ${\bf u}_+={\bf u}_-$, or else  for $r_+\neq r_-$ there is again a contact. Similarly, for the  case  $\Theta_+=-\Theta_-$,  $(s-\phi(-\Theta_-))r_+-(\phi_--s)r_-=0$, or $s=\frac{\phi_-r_-+\phi(-\Theta_-)r_+}{r_-+r_+}$, which produces an {\em anomalous} shock (see \ct{KK}). We restrict our attention to  $\Theta_+\neq-\Theta_-$ in the following  analysis.

To solve the Riemann problem, we consider three cases for left and right states, ${\bf u}_\pm$, with $\phi_{-}=\phi_+$ or $\phi_{-}<\phi_+$ or $\phi_{-}>\phi_+$. The first two  can be solved using contact, constant, and \lq cavitation\rq\, solutions, in a similar way to the two-dimensional case, (\ct{YZ2}), but with solutions  generally not  unique in higher dimensions, even under various forms of \lq monotonicity\rq\, assumptions. Defining

\be U_{\phi_0}=\{ {\bf u}\in\mathbb R^{n}: {\bf u}=r\Theta, \phi(\Theta)=\phi_0\},
\ee
$\bf u_+$ and $\bf u_-$  can be connected by a contact if and only if  ${\bf u}_\pm\in U_{\phi_0}$ for some $\phi_0$.

For $\phi_-<\phi_+$ let $\uu(x, t)=\uu(\xi), \xi=x/t$ and consider self-similar solutions in the wedge $\{(x, t): \phi_- t< x< \phi_+ t\}$. Equation \re{KK} for Riemann data \re{riemann}  reduces to the boundary value problem 
 \be
 (-\xi+\phi)(r\Theta)_\xi+\phi_\xi r\Theta=\bf{0}, \quad \uu(\pm\infty)=\uu_\pm,
 \ee
or,
  \be
 (\phi_\xi r + (\phi-\xi)r_\xi)\Theta+(\phi-\xi) r\Theta_\xi=\bf{0}.\eqn{th}
 \ee
 In the case that $\Theta_\xi={\bf 0}$, and so  $\phi_\xi=0$,   either $\phi=\xi$, a contradiction, or else $r_\xi=0$, in which case 
 $\uu_+=\uu_- $. 
 However, if $\Theta_\xi\neq{\bf 0}$, then since both coefficients in \re{th} must be zero, either $\phi=\xi$ and $\phi_\xi=0$, again a contradiction, or else $r=0$. Thus for $\uu_+\neq\uu_-$, provided that characteristics are non-overlapping ($\phi_-t<\phi_+t$),  connections between left and right states $\uu_\pm$ are via  \lq cavitation\rq\, ($\uu={\bf 0}$) in the region $\phi_-<\xi<\phi_+$, {\em i.e.} ${\bf u}(x, t)=r_-\Theta_-(1-H(x-\phi_-t))+ r_+\Theta_+H(x-\phi_+t)$.

As an example, if we consider the  case of $\phi(\Theta)$ having dependence on $q={\bf a}\cdot\Theta, |{\bf a}|=1$, with $\phi_q(q)>0$,
then for contacts, $\phi_-=\phi_+$,

\be {\bf u}_\pm\in U_{\phi_0} \iff \Theta_\pm\in S^{n-1}\cap U_{q_0}, \ee

\noindent where $U_{q_0}=\{ {\bf u}\in\mathbb R^{n}: {\bf u}=r\Theta, {\bf a}\cdot\Theta=q_0\}$ and $\phi(q_0)=\phi_0$. Such a case can be found in the context of geometrical optics for $n=2, {\bf{a}}=(1, 0)$ and $\phi(q)=q$ in (\ct{ER1}),
where for $\phi_-<\phi_+$,  cavitation solutions connect left and right states between corresponding sets $U_{q_-}$ and $U_{q_+}$ with $-1<q_-<q_+<1$, (see \ct{YZ2}). 

The absence of shocks is a  structural feature of the  $\phi(\Theta)$ system. Because the flux depends only on the angular variable, the system is invariant under radial rescaling of the state and the reduction to contact waves is not simply a consequence of repeated eigenvalues or linear degeneracy. It stems from the fact that the  variable $\Theta$  takes values on the nonlinear manifold $S^{n-1}$, where the evolution equation is a transport equation rather than a conservation law. Scale invariance renders the radial amplitude dynamically passive, forces complete linear degeneracy of the characteristic fields, and removes the constitutive mechanism responsible for genuinely nonlinear wave formation. As a result, there is no classical Rankine–Hugoniot mechanism for generating new admissible jump speeds. The hypothesis fixes the propagation speed pointwise, and discontinuities  only persist when both sides  possess the same characteristic speed. In this sense, the geometry of the state manifold itself replaces the usual entropy-shock mechanism, making the $\phi(\Theta)$ case a singular limit of the Keyfitz–Kranzer family.
  
 Finally, for $\phi_->\phi_+$,  characteristic fields coming from left and right states overlap and the classical Rankine-Hugoniot equations, \re{RH}, fail to have solutions.  Recalling Section  \ref{theta section}, where it was found that $r$ could become unbounded in finite time, we will consider this inequality in the context of unbounded solutions in Section \ref{deltatheta}.

\section{Delta-Shocks}

 In this section we consider weak solutions, ${\bf u}(x, t):\mathbb R\times\mathbb R_+\rightarrow{\cal B}\subset \mathbb R^n$, to the initial value problem for \re{KK}, 
 with possible singularities in $\uu$ resulting in $\cal B$  being  unbounded in general, and delta-shock solutions, $\delta_{\sigma}$,  propagating along  ${\SD}=\{(x, t)\in\mathbb R\times\mathbb R_+: x={x}(t)\}$.

 An integral balance law formulation for  system \re{KK}  is useful    both for bounded and for unbounded forms of generalized solutions. It will be assumed that any weak solution to \re{KK} with  possible discontinuities or concentrations in ${\bf u}(x, t)$ on a curve $\SD=\{(x, t)\in\mathbb R\times\mathbb R_+: x={x}(t), x(0)=0\}$, satisfies the relation
 \footnote{  Reynold's transport theorem can be directly verified  to hold for time-dependent  Radon measures, $\vv=\hat{\vv}(x, t)dx+{\bf w}(t)\delta(x-x(t)), \hat{\vv}\in L_{loc}^1, {\bf w}\in W^{1, \infty}$ -
 provided $ a(t)<x(t)<b(t)$,   
$\frac{d}{dt}\int_{a(t)}^{b(t)}\vv(x, t)dx=
\int_{a(t)}^{b(t)}\vv_t(x, t)dx+\vv(b(t),t)\frac{db}{dt}-\vv(a(t),t)\frac{da}{dt}$. }
\be
\frac{d}{dt}\int_{a(t)}^{b(t)}\uu(x, t)dx=
[\phi(\uu(\cdot,t))\uu(\cdot,t)]_{b(t)}^{a(t)}+\uu(b(t),t)\frac{db}{dt}-\uu(a(t),t)\frac{da}{dt}
\eqn{cal}
\ee
for every $-\infty<a(t)<x(t)<b(t)<\infty$.
In particular, letting $a(t)\rightarrow x^-(t)$ and $b(t)\rightarrow x^+(t)$,   solutions  across $\SD$   formally satisfy 

\be
\frac{d}{dt}\int_{x^-(t)}^{x^+(t)}\uu(x, t)dx =[\phi(\uu(\cdot,t))\uu(\cdot,t)]_{x^+(t)}^{x^-(t)} - [\uu(\cdot, t]_{x^+(t)}^{x^-(t)}\frac{dx(t)}{dt}.
\eqn{bKK}
\ee
Weak solutions, ${\bf u}(x, t):\mathbb R\times\mathbb R_+\rightarrow{\cal B}\subset \mathbb R^n$, to the Riemann problem for \re{KK}
  with $\cal B$  bounded, and admitting discontinuities propagating along   $\SD$  in which the left side of \re{bKK} vanishes, therefore allow   jumps   travelling at  the speed $\frac{dx}{dt}=s$,
 satisfying 
\re{RH}. 
 In general, we will write \re{bKK} as
 \be
 \frac{d}{dt}\int_{x^-(t)}^{x^+(t)}\uu(x, t)dx= (s-\phi (\uu_+))\uu_++(\phi (\uu_-)-s)\uu_- =\alpha_+\Theta_++\alpha_-\Theta_-,\nonumber\\
 \eqn{sbKK}
\ee
 where
\be \alpha_\pm=\pm (s-\phi(\uu_\pm)r_\pm,\quad s=\frac{dx}{dt}.
 \eqn{occ}
 \ee

In relation both to bounded and unbounded $\cal B$, we first make a few observations obtained  by connecting  points $(x, \tau)=(x(\tau), \tau), \tau>0$,  on $S_\sigma$ to initial data. Assuming $m_->\frac{dx(t)}{dt}>m_+$, and taking  $x_L(t)= x(\tau)-m_-(\tau-t)$ and $x_R(t)=x(\tau)-m_+(\tau-t)$, $0\leq t\leq\tau$,
\re{cal} gives,
\be
\frac{d}{dt}\int_{x_L(t)}^{x_R(t)}\uu(x, t)dx=(\phi_--m_-)\uu_--(\phi_+-m_+)\uu_+, \quad0\leq t\leq\tau,
\eqn{tra}
\ee
where $\uu_-=\uu(x_L (t), t)$ and $\uu_+=\uu(x_R(t), t)$,   $\phi_\pm$ correspondingly.   Then, given classical Riemann data  where states  to the left and right of $x=x(t)$ remain constant together with $\frac{dx(t)}{dt}\equiv s$, 
 we find
\begin{eqnarray}
I(t)_{m_{\pm}}\equiv&\int_{x_L(t)}^{x_R(t)}\uu(x, t)dx\nonumber\\
=&((\phi_--m_-)\uu_-+(m_+-\phi_+)\uu_+)t+\int_{x_L(0)}^{x_R(0)}\uu(x, 0)dx\nonumber\\
=&((m_--\phi_-)\uu_-+(\phi_+-m_+)\uu_+)(\tau-t)+((\phi_--s)\uu_-+(s-\phi_+)\uu_+)\tau.\quad
\eqn{cone}
\end{eqnarray}
That is, given arbitrary $\tau>0$, $I(t)_{m_{\pm}}$ satisfies
 $$I(0)_{m_{\pm}}=((m_--s)\uu_-+(s-m_+)\uu_+)\tau,$$ and tends to the  value $$I(\tau)_{m_{\pm}}=((\phi_--s)\uu_-+(s-\phi_+)\uu_+)\tau=(\alpha_-\Theta_-+\alpha_+\Theta_+)\tau\equiv I(\tau)$$ as $t$ approaches  $\tau$, where the right side is zero only if $\uu$ satisfies the Rankine-Hugoniot equations \re{RH}.
In particular, this construction becomes possible for $m_\pm=\phi_\pm$ provided the geometric entropy condition, $\lambda_-> s> \lambda_+$, holds, which means $\alpha_\pm>0$ and, by  \re{cone},   that  $I(t)_{\phi\pm}$ is conserved.
 Further, if we choose $m_\pm$ either  close to $s\pm$, or identical to $\phi_\pm$,  each  integration tends to the same final value, $I(\tau)$. Thus, since $I(t)_{\phi\pm}$ remains constant, we obtain from $x\in\Omega(t)\equiv (s\tau-\phi_-(\tau-t), s\tau-\phi_+(\tau-t))$ that for arbitrary $0\leq t\leq\tau$,
\be
I(t)_{\phi_\pm}=(\alpha_-\Theta_-+\alpha_+\Theta_+)\tau,
\eqn{conserve}
\ee
while we can   use \re{cone} directly to find the contribution from $x\in\Omega_\epsilon (t)\equiv(st-\epsilon(\phi_--s)(\tau-t), st-\epsilon(\phi_+-s)(\tau-t))$ to be
\be
I(t)_{s+\epsilon(\phi_\pm-s)}=(\alpha_-\Theta_-+\alpha_+\Theta_+)(t+\epsilon(\tau-t)).
\eqn{conelim}
\ee
So, setting 
\be
M_\epsilon(t)=\frac{|I(t)_{s+\epsilon(\phi_\pm-s)}|}{|I(t)_{\phi_\pm}|},
\ee
 $M_\epsilon(t)$ takes $M_\epsilon(0)=\epsilon$ to $M_\epsilon(\tau)=1$ as $t\rightarrow\tau$, which means that the inner region,  with measure 
 \begin{eqnarray}
 A_{\epsilon\phi_\pm} (t)\equiv \epsilon(\phi_--\phi_+)(\tau-t)
 \equiv \epsilon A_{\phi\pm}(t),
 \eqn{hel}
 \end{eqnarray}
  collapses to the point $x=s\tau$ as $t\rightarrow \tau$ while increasingly   contributing to $I(t)_{\phi\pm}$ as $t$ grows, in contrast to the fixed contribution to $I(t)_{\phi\pm}$ from $A_{\phi\pm}(t)$.
 Equivalently, the outer region with measure $(1-\epsilon)A_{\phi\pm}(t)$
 contributes
\be
J_\epsilon(t)=I(t)_{\phi_\pm}-I(t)_{s+\epsilon(\phi_\pm-s)}
\eqn{JJ}
\ee
whose magnitude decreases, with $J_\epsilon(0)=(1-\epsilon)I(\tau)$ and $J_\epsilon(t)\rightarrow {\bf 0}$ as $t\rightarrow\tau$ while 
the magnitude of $I(t)_{s+\epsilon(\phi_\pm-s)}$ increases from $\epsilon I(\tau)$  to $I(\tau)$, and
\be
J_\epsilon(t)+I(t)_{s+\epsilon(\phi_\pm-s)}=I(\tau), \quad 0\leq t<\tau,
\eqn{J+}
\ee
for arbitrary $\epsilon>0.$ 

Next, set $\varrho=|\alpha_-\Theta_-+\alpha_+\Theta_+|$,\footnote{Note  here, that with Riemann data, the vector quantities in \re{conserve}, \re{conelim} and \re{JJ}, \re{J+} are each aligned with $\alpha_-\Theta_-+\alpha_+\Theta_+$, and that if $\phi=\phi(z)$, no concentration can occur in z.} 
and denote $x_{R/L}(t)=s\tau-\phi_{\pm}(\tau-t)$ and $x_{R\epsilon/L\epsilon}(t)=st-\epsilon(\phi_\pm-s)(\tau-t)$.
Since $|I(t)_{\phi\pm}|=\varrho\tau$ is constant, we can  define a Radon measure 
$i_{\phi\pm}(t)(dx)$
 so that
$$<i_{\phi\pm}(t), \psi>=\varrho\tau A_{\phi\pm}(t)^{-1}\int_{x_L(t)}^{x_R(t)}\psi dx(t),$$ 
for all $\psi\in C_0(-\infty, \infty)$ and, since $A_{\phi\pm}(t)\rightarrow 0$ and 
$\sup_{\Omega(t)}|\psi (x, t) - \psi (st, t)|\rightarrow 0$ as $t\rightarrow \tau$, therefore $i_{\phi\pm}(t)$ converges weak$-*$ to a Dirac measure,
$$i_{\phi\pm}(t)\rightharpoonup\varrho\tau\delta(x-s\tau).$$ 
Similarly, since \re{conelim} implies that
 $$\frac{|I(t)_{s+\epsilon(\phi_\pm-s)}|}{t+\epsilon(\tau-t))}$$
remains constant we can, in turn, 
 construct the measure
$i_{s+\epsilon(\phi_\pm-s)}(t)(dx)$
where
$$<i_{s+\epsilon(\phi_\pm-s}(t), \psi>=(\varrho(t+\epsilon(\tau-t)) A_{\epsilon\phi\pm}(t)^{-1}\int_{x_{L\epsilon}(t)}^{x_{R\epsilon}(t)}\psi dx(t)$$ for all $\psi\in C_0(-\infty, \infty)$, {\em ie.} 
$$i_{s+\epsilon(\phi_\pm-s)}(t)\rightharpoonup\varrho\tau\delta(x-s\tau),$$ as $t\rightarrow\tau$, while for $0\leq t<\tau$,
$$
i_{s+\epsilon(\phi_\pm-s)}(t)\rightharpoonup\varrho t\delta(x-st)
$$
as $\epsilon\rightarrow 0.$

Combining the results above, we have a two-scale decomposition representation for $\uu$ as a vector-valued measure over the region $\Omega(t)$, obtained by splitting the inner and outer parts, $i_{s+\epsilon(\phi_\pm-s)}(t)(dx)$ and $(i_{\phi\pm}-i_{s+\epsilon(\phi_\pm-s)})(t)(dx)$,  into pairwise orthogonal components 
 in which the total mass is conserved, while transferring from the outer to the inner region with both contracting as t increases, {\em ie.}
 \be
 {\bf u}(t)(dx)=(\alpha_-\Theta_-+\alpha_+\Theta_+)((i_{\phi\pm}-i_{s+\epsilon(\phi_\pm-s)})+i_{s+\epsilon(\phi_\pm-s)})(t)(dx).
 \ee
 In particular, as $t\rightarrow\tau$,
 \be
 \uu(t)(dx)\rightharpoonup(\alpha_-\Theta_-+\alpha_+\Theta_+)\tau\delta(x-s\tau),
 \ee
 while as $\epsilon\rightarrow 0,$
 \be
 \uu(t)(dx)\rightharpoonup (\alpha_-\Theta_-+\alpha_+\Theta_+)t\delta(x-st),\, t<\tau.
 \eqn{star}
 \ee

In a slightly more general setting, single $\delta-$shock solutions to \re{bKK} can be specified as

\be
\uu(x, t)=\hat{\uu}(x, t)+{\bf w}(t)\delta_{\sigma} 
\eqn{delta}
\ee
where $\hat{\bf u}(., t)\in L^\infty$, the \lq singular mass',  ${\bf w}(t)$, lies in $W^{1, \infty}[0, T), T>0,$ and $\delta_{\sigma} =\delta (x-x(t))$ denotes a Dirac measure centered on a corresponding curve ${\SD}$.
Provided  that ${\bf u}$  has one-sided limits ${\bf u}_{\pm}  \mbox{ }(=\hat{\uu}_\pm)$ at ${\SD}$, substituting \re{delta} into \re{occ} then gives  the generalized Rankine-Hugoniot relation,

\be
\frac{d{\bf w}}{dt}=[\phi({{\bf u}}){{\bf u}}]-[{\bf u}]\frac{dx}{dt}
\eqn{dshock}
\ee
which, for constant left and right states and speed $\frac{dx}{dt}=s$, reads 
 \be
\dot{\bf w}(t)=\alpha_+\Theta_++\alpha_-\Theta_-,
\eqn{dRshock}
\ee
where the $\alpha_\pm$  are nonnegative for overcompressive data, and we now assume that generally $\Theta_+\neq\pm\Theta_-$ from the remark following \re{uu-+2}. The right side of  equation \re{dRshock}  represents a Rankine-Hugoniot  deficit.

For  convenience, we   refine the solution representation  \re{delta} by formally expressing it in  \lq polar\rq\,  form as
\be
\uu (x, t)=(\hat{r}(x, t)+w(t)\delta_{\sigma}) ((1-\breve\delta_{\sigma})\hat{\Theta}(x, t))+\breve\delta_{\sigma}{\Theta}_{\sigma} (t)),
\eqn{new}
\ee
where $\Theta_{\sigma}(t)\in S^{n-1}$   and $\breve{\delta}_{\sigma}=\breve\delta(x-x(t))$, with $\breve\delta(.)\in L^\infty$  given by
$$
\breve\delta (x)=0, \mbox{ if} \, x\neq 0, \quad \breve\delta (0)=1.
$$
The functions $\hat{r}(x, t)$ and $\hat{\Theta}(x, t)$ are  defined
by
\be
\hat{r}(x, t)=r_--[r]H(x-x(t)), \quad
\hat{\Theta}(x, t)=
\Theta_--[\Theta]\,H(x-x(t)),
\eqn{Thetahat}
\ee
 corresponding to   initial data  $\uu_0$ with piecewise constant states ${r}_\pm\Theta_\pm$ for $x$ less than, or greater than, $x(0)=0$. 
 With ${\bf u}$ having  left and right  states  at ${\SD}$ defined by \re{Thetahat}, integrating \re{new}, or equivalently \re{delta}, across ${\SD}$ implies that

\be
{\bf w}(t)=\int_{x(t)^-}^{x(t)^+}{\bf u}\, dx=w(t)\,\Theta_{\sigma}(t)
\eqn{intsigma}
\ee
under the assumption that $\int_{x(t)^-}^{x(t)^+} \delta_{\sigma}\,\breve\delta_{\sigma} dx=1$,
\footnote{This is consistent with the  Dirac measure $\delta_{\cal S}$ being treated as a measure acting on pointwise-defined functions and interpreting the Radon measure $\nu = \delta_{\cal S}\breve\delta_{\cal S}$ as being defined by integration against $\delta_{\cal S}$.}
and, for initial data having distributional component ${\bf w}(0)=w(0)\Theta_{\sigma} (0)$ at $x(0)=0$, integrating \re{dRshock} and using \re{intsigma} gives
\be
w(t)\Theta_{\sigma}(t)=w(0)\Theta_{\sigma}(0)+(\alpha_+\Theta_++\alpha_-\Theta_-)t,
\eqn{w}
\ee
which, in the  case of standard Riemann  data with $w(0)=0$, reduces \re{new} to
\be
\uu (x, t)=(\hat{r}(x, t)+\Rho\,t\,\delta_{\sigma})((1-\breve\delta_{\sigma})\hat{\Theta}(x, t))+\breve\delta_{\sigma}{\Theta}_{\sigma}))
\eqn{newer}
\ee
where $\Rho\equiv |\alpha_+\Theta_++\alpha_-\Theta_-|, w(t)=\Rho t$ while, for $x(0)=w(0)=0$, equation \re{w} implies that $\Theta_{\sigma}(t)$ satisfies
\be
\Rho\Theta_{\sigma}=\alpha_+\Theta_++\alpha_-\Theta_-,\quad \alpha_\pm>0,
\eqn{now}
\ee
(see \re{star})
and so $\Theta_{\sigma}\in S^{n-1}\cap sp\{\Theta_-, \Theta_+\}$ is constant and lies between $\Theta_+$ and $\Theta_-$ on the great circle which they generate. If $\Rho=0, \, \mbox{ then either\,}\, \alpha_\pm=0$ and  we recover {\em contacts},  $s=\phi_+=\phi_-  $ (for $r_\pm>0$), or else $\Theta_+, \Theta_-\in S^{n-1}$ are linearly dependent, and   we obtain {\em Lax} shocks for $\Theta_+=\Theta_-$ (with $\alpha_+\alpha_-<0$, $s=\frac{\phi_+r_+-\phi_-r_-}{r_+-r_-}$), or {\em anomalous} shocks for $\Theta_+=-\Theta_-$ (with $\alpha_+\alpha_- >0$,  $s=\frac{\phi_+r_++\phi_-r_-}{r_++r_-}$).

\subsection{The Case  $\phi=\phi(z)$}\label{sec:zdelta}
Here we refine the form of distribution solution \re{delta} somewhat further. As a result of equation \re{z1}, or repeating the earlier analysis in this chapter using $z$ in place of $\uu$, $z$ should not possess a Dirac mass on ${\SD}$ (we will  verify this subsequently when we adopt  \re{newer} as a particular type of solution) and so, since $z(x, t)=|\hat{\uu}(x, t)+{\bf w}(t)\delta_{\sigma} |\,K(\Theta (x, t))$ for solutions of the form \re{delta}, we must have $K(\Theta (x(t), t))=0,$
where the curves $\Theta_{\sigma}(t)\equiv\Theta (x(t), t)$ are assumed to lie in the intersection of $ S^{n-1}$ and  $n-2$ dimensional hypersurface(s), $\cal H$, across which, recalling the nondegeneracy hypothesis \re{nond} (and accordingly $\nabla_\vartheta\KK (\vartheta_{\sigma} (t))\neq {\bf 0}$), $\KK(\vartheta)$ changes sign, putting us in the overcompressible case (see \re{com}). 
\footnote{If $n=2$, each surface $\cal H$ reduces to a curve intersecting the unit circle   transversely at points, under the nondegeneracy condition $\KK'(\vartheta_{\cal S})\neq 0$,  while for $n=3$ the intersection reduces to a curve(s) on $S^2$.}

Since $z$ has no dirac mass on ${\cal S}$,   \re{z1} implies that   discontinuities    move at   speed $s$, 
\be
s=\frac{[\phi({z}){z}]}{[{z}]}, \quad {z}={r}K({\Theta}),
\eqn{s}
\ee
according to which (see \re{zKK}) $s$ satisfies  the relation 
\be
0=\alpha_+K(\Theta_+)+\alpha_-K(\Theta_-),
\eqn{zs}
\ee

\noindent from which we recover  $K(\Theta_-)\geq 0\geq K(\Theta_+)$, with $K(\Theta_{\sigma})=0$, stemming from the inequality $\lambda_-\geq s\geq \lambda_+$.

\begin{lemma}
Let $K(\Theta_-)>0>K(\Theta_+)$. Then $\Theta_\sigma$ is given by
\be
\Theta_\sigma=\frac{K_-\Theta_+-K_+\Theta_-}{|K_-\Theta_+-K_+\Theta_-|},
\eqn{key}
\ee
provided $\Theta_+$, $\Theta_-$ satisfy the relationship
\be
K(\frac{K_-\Theta_+-K_+\Theta_-}{|K_-\Theta_+-K_+\Theta_-|})=0.
\eqn{zero}
\ee
\end{lemma}
\proof
Recalling \re{srz} and as in \re{uu-+2}, we find
\be
\Rho\Theta_{\sigma}= \frac{(\phi_+-\phi_-)r_+r_-}{r_+K_+-r_-K_-}(K_-\Theta_+-K_+\Theta_-),
\eqn{newr}
\ee
 using \re{now},
{\em i.e.} for $K(\Theta_-)>0>K(\Theta_+)$,  we have
\be
\Rho=\frac{[\phi]}{[z]}r_+r_-|K_-\Theta_+-K_+\Theta_-|\quad \mbox{and}\quad\Theta_{\sigma}=\frac{K_-\Theta_+-K_+\Theta_-}{|K_-\Theta_+-K_+\Theta_-|}.
\eqn{nophi}
\ee
Finally, requiring $K(\Theta_{\sigma})$ to be zero, 
gives \re{zero}.
\qed
\begin{remark}
A value can  be formally assigned to $z_\sigma$ via the use of \re{KKb}. Since $\phi(z_-)>s>\phi(z_+)$ and $\phi'(z)>0$, in order for  $\Theta_\sigma$ to propagate
with speed $s\equiv\phi(z_\sigma)$, implies  $z_\sigma=\phi^{-1}(\frac{[\phi(z)z]}{[z]})$, where $z_->z_\sigma>z_+$.
\end{remark}

In summary, we have a representation formula.

\begin{theorem}
Let $\uu (x, t)$ be any solution to  \re{KK} with bounded Riemann data,   \re{Thetahat}, (\re{riemann}),  
of the form \re{newer} for $\phi =\phi(z)$, i.e.
\be  \uu (x, t)=(\hat{r}(x, t)+\Rho t\delta_{\sigma}) ((1-\breve\delta_{\sigma})\hat{\Theta}(x, t))+\breve\delta_{\sigma}{\Theta}_{\sigma})),
\nonumber
\ee
for which
\be
K(\frac{K_-\Theta_+-K_+\Theta_-}{|K_-\Theta_+-K_+\Theta_-|})=0.
\eqn{harder}
\ee
 Then $\bf{u}(x, t)$ satisfies the balance law \re{bKK} for $x=st$, 
\be
\frac{d}{dt}\int_{st^-}^{st^+}\uu(x, t)dx - s[\uu(\cdot, t]_{st^-}^{st^+}+[\phi(z(\cdot,t))\uu(\cdot,t)]_{st^-}^{st^+}={\bf 0},
\eqn{again}
\ee
where
\be
s=\frac{[\phi(z)z]}{[z]}.
\ee
\end{theorem}

We now note two special cases, which avoid significantly restricting  the forms of $\phi$ or $K$ (see Shen \ct{Shen}).
\begin{lemma}
Let $K(\Theta_-)>K(\Theta_+)=0$, or $K(\Theta_+)<K(\Theta_-)=0$. Then $\Rho\Theta_{\sigma}=\phi(z_-)r_+\Theta_+$, or $\Rho\Theta_{\sigma}=-\phi(z_+)r_-\Theta_-$, respectively.
\end{lemma}
\proof
If $K(\Theta_-)>0$, $K(\Theta_+)=0$,   then by \re{zs},  $\alpha_-=0$,  $s=\phi(z_-)$ and  $\Theta_{\sigma}=\Theta_+$, and since $\phi(z_+)=\phi(0)=0$, \re{now}  gives
$$\Rho=\alpha_+=\phi(z_-)r_+.$$
 
\noindent Similarly, if $K(\Theta_-)=0$, $K(\Theta_+)<0$ , so by \re{zs}, $\alpha_+=0$,  $s=\phi(z_+)$, and then $\Theta_{\sigma}=\Theta_-$. \re{now} now gives
$$
\Rho=\alpha_-=-\phi(z_+)r_-.
$$
In     both cases, equation \re{zero} is trivially satisfied.

\qed

In general, for $K(\Theta_-)>0>K(\Theta_+)$, the function $K(\Theta)$ affects under which circumstances  the states ${\bf u}_\pm$ satisfy  \re{now}, \re{zs} and \re{zero}.
The most straightforward case,  permitting general left and right states, occurs if $K(\Theta)$ is  an inner product, ${\bf a}\cdot\Theta$, for some nonzero element ${\bf a}\in \RR^{n}$.
This situation appears in work of Shelkovich \ct{S}, Yang and Zhang \ct{YZ1}, and Cheng and Yang \ct{CY1}, which motivates the following.

 \begin{lemma}
 Suppose $\Theta_{\sigma}, \Theta_\pm\in S^{n-1}$  satisfy equation  \re{now} for  $\Rho>0$ and $\alpha_\pm> 0$, 
 together with  \re{zs} for $K(\Theta_-)> 0> K(\Theta_+)$ where \mbox{$K(\Theta)={\bf a}\cdot\Theta$} and ${\bf a}\in S^{n-1}$.
  If $\Theta\in sp\{\Theta_- , \Theta_+\}$, and $\Theta_\sigma$ is given by \re{nophi}, then   $K(\Theta)={\bf \Theta_{\sigma}^\perp\cdot\Theta}$ for some nonzero  projection $\Theta_{\sigma}^\perp\in sp\{\Theta_- , \Theta_+\}$,  $K(\Theta_\sigma)=0$, and condition \re{zero} holds identically.
 \end{lemma}

 \proof

 Recall that if $\Theta_\pm$ are linearly dependent, then $\Theta_+=-\Theta_-$ by \re{zs} since $K(.)$  is odd and  $\alpha_\pm>0$,  finally implying that $\alpha_+=\alpha_-$. 
 As a result, equation \re{now}  gives $\Rho\Theta_{\sigma}=\alpha_+( -\Theta_-) +\alpha_- \Theta_- = {\bf 0}$ and so $\Rho=0$.
   Thus $\Rho>0$  means $\Theta_\pm$ must be linearly independent. Letting $\Theta_{\sigma}^\perp$
   denote the orthogonal projection of ${\bf a}$ onto $sp\{\Theta_- , \Theta_+\}$,   gives  ${\bf a}\cdot\Theta=\Theta_{\sigma}^\perp\cdot\Theta$ for all $\Theta\in sp\{\Theta_-, \Theta_+\}$. Since $K(\Theta_\pm)\neq 0$,   ${\bf a}$ cannot be perpendicular to $sp \{\Theta_- , \Theta_+\}$, and so 
 $\Theta_{\sigma}^\perp$ is not null and $K(\Theta)=\Theta_{\sigma}^\perp\cdot\Theta$.
 Verifying condition \re{zero} is immediate, since  $K(K_-\Theta_+-K_+\Theta_-)=0$ for $K$ linear.
 
  \qed
 
 This means we are  able to solve equations \re{now} and \re{zero}, connecting arbitrary left and right states  under the conditions $K_->0>K_+$ when $K={\bf a}\cdot\Theta$,
 and we note in passing the form of $\Theta_{\sigma}^\perp\in sp\{\Theta_-, \Theta_+\}$,
 \be
 \Theta_{\sigma}^\perp = 
c \,\frac{(K_+- K_-\Theta_+\cdot\Theta_-)\Theta_++(K_--K_+\Theta_+\cdot\Theta_-)\Theta_-}{|(K_+-K_-\Theta_+\cdot\Theta_-)\Theta_++(K_--K_+\Theta_+\cdot\Theta_-)\Theta_-|}\,,
\eqn{perp}
\ee
for some $|c|\leq 1$.

 Now we  remove the assumption of linearity for $K$  and make the hypothesis that  $K(\Theta)=\tilde{K}(q)$ for some  $C^1$ function, $\tilde{K}$, of $q={\bf a}\cdot\Theta,$ with $ |\bf{a}|=1$. The antipodal property \re{odd} for $K(\Theta)$ translates to $\tilde{K}(q)$ being odd, which we supplement with monotonicity.

 \begin{lemma}
 Suppose $\Theta_{\sigma}, \Theta_\pm\in S^{n-1}$  satisfy equation  \re{now} for  $\Rho>0$ and $\alpha_\pm> 0$, 
 together with  \re{zs} for $K(\Theta_-)> 0> K(\Theta_+)$. Let \mbox{$K(\Theta)=\tilde{K}(q)$}, $q={\bf a}\cdot\Theta$, $|\bf{a}|=1$,  $\tilde{K}(q)\in C^1$ odd and $\tilde{K}'(q)>0, q\in(-1, 1)$.  
  If $\Theta_\sigma$ is given by \re{nophi}, then    condition \re{zero} holds provided the states $\Theta_-$, $\Theta_+$ are connected by 
  \be
  \frac{\tilde{K}(q_-)}{q_-}
  = \frac{\tilde{K}(q_+)}{q_+}.
  \eqn{genhug}
  \ee
 \end{lemma}
 \proof
 Since $K(\Theta_-)> 0> K(\Theta_+)$ implies $\tilde{K}(q_-)>0>\tilde{K}(q_+)$, our assumptions mean that $q_\pm\neq 0$ and, in particular, ${\bf a}$ is not  normal to $sp\{\Theta_-, \Theta_+\}$.
  For $q_\sigma={\bf a}\cdot\Theta_\sigma$, 
 we  must however seek  sets of pairs $\{q_-, q_+\}\in(-1, 1)^2$ for which $K(\Theta_\sigma)=\tilde{K}(0)=0$.
But, taking the inner product of equation \re{key} with ${\bf a}$,   gives
\be
q_\sigma=\frac{\tilde{K}(q_-)q_+-\tilde{K}(q_+)q_-}{|\tilde{K}(q_-)\Theta_+-\tilde{K}(q_+)\Theta_-|},
\eqn{keyq}
\ee
which is zero by \re{genhug}, and the result follows.

 \qed

\begin{corollary}
Let $\Theta_\sigma$ and $\tilde{K}(.)$ satisfy the conditions of the Lemma above. Then $q_+=-q_-$, and
\be
\Theta_\sigma=\frac{\Theta_++\Theta_-}{|\Theta_++\Theta_-|}.
\ee
\proof
Adopting condition \re{odd} means that $\frac{\tilde{K}(q)}{q}$ is even, and so \re{genhug} implies $q_+=\pm q_-$, but $q_+=q_-$ implies $\tilde{K_+}(q)=\tilde{K}(q_-)$, so $\tilde{K}(q)$ does not change sign. Therefore $q_+=-q_-$and, from \re{key}, which here becomes

\be
\Theta_\sigma=\frac{\tilde{K}(q_-)\Theta_+-\tilde{K}(q_+)\Theta_-}{|\tilde{K}(q_-)\Theta_+-\tilde{K}(q_+)\Theta_-|},
\eqn{keyq}
\ee
together with $\tilde{K}(q_+)=-\tilde{K}(q_-)$, we obtain

\be
\Theta_\sigma=\frac{\tilde{K}(q_-)(\Theta_++\Theta_-)}{|\tilde{K}(q_-)||\Theta_++\Theta_-|}=\frac{\Theta_++\Theta_-}{|\Theta_++\Theta_-|}
\ee
since $\tilde{K}(q_-)>0$.
\qed

\end{corollary}

\begin{remark}

Although here it is necessary   to satisfy $q_+=-q_-$ for a given state $\Theta_-$ to be connected to a  choice of  possible  $\Theta_+$ by a delta shock, it is not sufficient. In particular, $\Theta_\sigma$ is  undefined under \re{odd} for $\Theta_+=-\Theta_-$, yet no delta shock  is  present, since $\Rho=0$ by \re{nophi}. Clearly, monotonicity of $\tilde{K}$ also requires all such delta shock   connections to satisfy $q_->0$.
\end{remark}

 
\begin{remark} By \re{intsigma}, \re{now}, we observe that

\be
\int_{st^-}^{st^+}{\bf u}\, dx=\Rho\,\Theta_{\sigma}\,t=((s-\phi_+)r_+\Theta_++(\phi_--s)r_-\Theta_-)\,t
\eqn{intsigma2}
\ee
and so the ansatz \re{newer} satisfies the requirement of the self-similar formula \re{uKK}. Further, equation \re{now} also lets us verify that $z(x, t)$ satisfies \re{z-+} since $K(\Theta_{\sigma})=0$, 

\be
\int_{st^-}^{st^+}z\,dx=\int_{st^-}^{st^+}w(t)K(\Theta_{\sigma})\delta (x-st)\breve\delta (x-st)\,dx=0
\eqn{zdx}
\ee
and finally,
\be
\int_{st^-}^{st^+}r\,dx=\int_{st^-}^{st^+}w(t)\delta (x-st)\breve\delta (x-st)\,dx=\Rho\,t.
\eqn{rdx}
\ee
 \end{remark}
 \subsection{The Case $\phi=\phi(\Theta)$}
 \label{deltatheta}

 Recalling equation \re{now},  we must now obtain a pair $(s, \Theta_{\sigma})$ satisfying
\be \Rho\,\Theta_{\sigma}=(s-\phi_+)r_+\Theta_++(\phi_--s)r_-\Theta_-,
\eqn{ful}
\ee
for which  $s=\phi(\Theta_{\sigma})$,
 and in order to have  an admissible delta-shock in this  case, we need to obtain such a solution subject to  the  inequalities
 \be
 \phi_-> s>\phi_+.
 \eqn{ent}
 \ee
Setting ${\mathcal F}{(\Theta)}=\Rho^{-1}((\phi(\Theta)-\phi_+)r_+\Theta_++(\phi_--\phi(\Theta))r_-\Theta_-)$, we  look for a fixed point, $\Theta=\Theta_{\sigma}$, to the equation
 \be
 \Theta=\mathcal{F}(\Theta)
 \eqn{fix}
 \ee
 subject to \re{ent} with $\Theta$ and $\mathcal{F}(\Theta)\in sp\{\Theta_-, \Theta_+\}\cap S^{n-1}$. Noting that since $sp\{\Theta_-, \Theta_+\}$ is either one or two-dimensional, the one-dimensional case having been shown  in Section \ref{sec:thetanice} not to contain any Dirac mass, we continue under the assumption that $\Theta_-\cdot\Theta_+\neq\pm 1$. Under this assumption, we have the following  for continuous $\phi(\cdot)$.
 \begin{lemma}
 Subject to the inequalities \re{ent}, there exists a fixed point, $\Theta_\sigma$, to \re{fix}.
 \end{lemma}
 \proof
We introduce a parametrized vector field, $\Theta(\gamma)\in sp\{\Theta_-, \Theta_+\}\cap S^{n-1}$, employing a pair of nonnegative, smooth, functions, $p_\pm(\gamma)$, such that 

\be
\Theta(\gamma)=p^{-1}(\gamma)(p_-(\gamma)\,\Theta_-+p_+(\gamma)\,\Theta_+),\quad\gamma\in [\gamma_-, \gamma_+],
\eqn{fixl}
\ee
 where $p^2={p_-}^2+2p_-p_+\Theta_-\cdot\Theta_++p_+^2$, $0<\gamma_+-\gamma_-<\pi$ and $|\Theta_-\cdot\Theta_+|<1$. Suppose  that
 $p'_-(\gamma)<0<p'_+(\gamma)$ and $p(\gamma)>0$ for all $\gamma\in(\gamma_-, \gamma_+)$,  while  $p_-(\gamma_-), p_+(\gamma_+)$ are nonzero with $p_-(\gamma_+)=p_+(\gamma_-)=0,$ so that 
 $$\Theta(\gamma_-)=\Theta_-, \quad\Theta(\gamma_+)=\Theta_+\quad \phi(\Theta_-)>\phi(\Theta_+).$$ 
 The conditions ensure that $\Theta_\gamma(\gamma)\neq {\mathbf 0}$.
Next, using the definition of $\mathcal F$ above,
\be
\mathcal{F}(\Theta(\gamma))=\Rho^{-1}(\gamma)(\alpha_-(\Theta(\gamma))\Theta_-+\alpha_+(\Theta(\gamma))\Theta_+), \quad\gamma\in [\gamma_-, \gamma_+],
\eqn{fixr}
\ee
where $\alpha_-(\Theta(\gamma))=(\phi_--\phi(\Theta(\gamma)))r_-$,   $\alpha_+(\Theta(\gamma))=
(\phi(\Theta(\gamma))-\phi_+)r_+$,  
\linebreak $\phi_\pm=\phi(\Theta_\pm)$, and 
$\Rho^2={\alpha_-}^2+2\alpha_-\alpha_+\Theta_-\cdot\Theta_++\alpha_+^2.$
 In particular, at $\gamma=\gamma_\pm$,
$$\mathcal{F}(\Theta(\gamma_-))=\Theta_+,\quad \mathcal{F}(\Theta(\gamma_+))=\Theta_-.$$
Thus $\gamma=\gamma_\pm$ are not  fixed points for \re{fix}.

Next, by using \re{fixl} and \re{fixr}  and consecutively taking inner products of \re{fix} with   elements $\Theta_{\pm\perp}\in sp\{\Theta_-, \Theta_+\}\cap S^{n-1}$ orthogonal to $\Theta_\pm$, means we require a solution to the $\gamma-$dependent system

\be
p^{-1}p_-=\Rho^{-1}\alpha_-,\quad  p^{-1}p_+=\Rho^{-1}\alpha_+,\quad \gamma\in (\gamma_-, \gamma_+),
\ee
in order to find a fixed point. Given that $\phi_->\phi_+$ and $r_\pm>0$, it follows that $\alpha_\pm>0$ and so,  since $p_\pm>0$,
we  have to show there exists $\gamma\in (\gamma_-, \gamma_+)$ such that $f(\gamma)=0$, where

\be
f(\gamma)\equiv p_-\alpha_+-p_+\alpha_-.
\ee
However, since $f(\gamma_-)=(\phi_--\phi_+)r_+>0$ and $f(\gamma_+)=-(\phi_--\phi_+)r_-<0$, the Intermediate Value theorem supports at least one value of $\gamma\in (\gamma_-, \gamma_+)$ where $f(\gamma)=0$,  which delivers a solution to \re{fix}. Local, or global, uniqueness follows. Assuming, for instance, $\phi(\Theta(\gamma))_\gamma<0$ for all $\gamma\in (\gamma_-, \gamma_+)$, $0<\gamma_+-\gamma_-<\pi$, global uniqueness can be found by computing $f'(\gamma)$ for  $\phi_->\phi_+$. In this case, using the  conditions stated above for $p'_\pm$, means  $f'(\gamma)={p'}_{-}\alpha_+ -{p'}_{+}\alpha_- +\phi_\gamma (p_-r_++p_+r_-)<0$, so that no further zeros of $f(\gamma)$  occur. \qed

\bigskip
Given the existence of $\Theta_\sigma$, we introduce  the   element $\Theta_{\sigma\perp}\in sp\{\Theta_-, \Theta_+\}$, perpendicular to $\Theta_\sigma$, which satisfies  $\Theta_{\sigma\perp}\cdot\Theta_->0>\Theta_{\sigma\perp}\cdot\Theta_+$ (recall that $\Theta_\sigma$ lies in the wedge  between $\Theta_-$ and $\Theta_+$). If we set $q=\Theta_{\sigma\perp}\cdot\Theta$ and take the inner product of $\Theta_{\sigma\perp}$ with both sides of equation \re{ful},   this implies 

\be
0=(s-\phi_+)r_+q_++(\phi_--s)r_-q_-, 
\eqn{fin}
\ee
or,
\be
s=\frac{\phi(\Theta_-)q_-r_--\phi(\Theta_+)q_+r_+}{q_-r_--q_+r_+},
\eqn{fin2}
\ee
that is, $s=\phi(\Theta_\sigma)=\frac{[\phi qr]}{[qr]}$.
The amplitude, $\Rho$,  then follows  by substituting  $s$ into \re{ful},   leading to

\be
\Rho=\frac{[\phi]}{[qr]}(r_-r_+)\{q_-^2-2q_-q_+\Theta_-\cdot\Theta_++q_+^2\}^{1/2},
\eqn{fin3a}
\ee
or as a result, 
\be
\Rho=\frac{[\phi]}{[qr]}r_-r_+\sin\beta, 
\eqn{fin3b}
\ee
where $\beta=\gamma_+-\gamma_-$ denotes the interior angle  between $\Theta_-$ and $\Theta_+$.

\bigskip
Although the results so far have  established the existence of $\Theta_\sigma$,    we yet have to find a   formula for $\gamma_\sigma$ (where $\Theta_\sigma=\Theta(\gamma_\sigma)$), and this turns out to be explicit only in particular cases. To obtain these, with $0<\gamma_+-\gamma_-<\pi$ (see \re{ful}, \re{ent}) and $\phi_->\phi_+$,
 we first  set $\gamma_{\sigma\perp}=\gamma_\sigma-\pi/2<\gamma_-<\gamma_\sigma<\gamma_+,$   in which case $q=\cos(\gamma_{\sigma\perp}-\gamma)=\cos(\gamma_\sigma-\pi/2-\gamma)=\sin(\gamma_\sigma-\gamma)$, $q_->0>q_+$. 
Inserting  $q_\pm=\cos\gamma_\pm\sin\gamma_\sigma-\sin\gamma_\pm\cos\gamma_\sigma$ into \re{fin2} now gives

\be
s=\frac{[\phi r\cos\gamma]\sin\gamma_\sigma-[\phi r\sin\gamma]\cos\gamma_\sigma}{
[r\cos\gamma]\sin\gamma_\sigma-[r\sin\gamma]\cos\gamma_\sigma}.
\eqn{fin3}
\ee
Next, we choose a (data dependent) orthonormal basis\footnote{$\bf{e}=(\sin\beta)^{-1}(\Theta_+\sin\gamma_--\Theta_-\sin\gamma_+), \bf{f}=(\sin\beta)^{-1}(\Theta_+\cos\gamma_--\Theta_-\cos\gamma_+)$.} $\{\bf{e}, \bf{f}\}$ for $sp\{\Theta_-, \Theta_+\}$  in order to write $\Theta(\gamma)$ 
as $\cos(\gamma){\bf e}+\sin(\gamma){\bf f}, \gamma\in (\gamma_-, \gamma_+)$   or, equivalently, in the form $\Theta(\gamma) = {(1+\tau^2)^{-1/2}}({\bf e}+ \tau{\bf f})$, where $\tau=\tan\gamma$.  We   then denote $\phi(\Theta(\gamma))$  by $\tilde{\phi}(\tau)$,   so that $\tilde{\phi}(\tau_\sigma)=s$ where $\tau_\sigma\in\mathbb{R}$ can be sought by solving
\be
\phi({(1+\tau_\sigma^2)^{-1/2}}({\bf e}+ \tau_\sigma{\bf f}))=
\tilde{\phi}(\tau_\sigma)=\frac{[\phi r\cos\gamma]\tau_\sigma-[\phi r\sin\gamma]}{[r\cos\gamma]\tau_\sigma-[r\sin\gamma]}
\equiv\frac{[\phi r\Theta]\cdot(\tau_\sigma, -1)}{[r\Theta]\cdot(\tau_\sigma, -1)}.
\eqn{fin4}
\ee
However in  many cases we can equivalently use a \lq projective\rq\, tangent $\tau_{ij}=\Theta_i/\Theta_j, i\neq j$ to parametrize $\Theta(\gamma)$(see~Appendix \ref{A}), in which \re{fin4} can more simply be replaced with
\be
\tilde{\phi}(\tau_{ij\sigma})=\frac{[\phi r\Theta_j]\tau_{ij\sigma}-[\phi r\Theta_i]}
{[r\Theta_j]\tau_{ij\sigma}-[r\Theta_i]}\equiv\frac{[\phi r(\Theta_j, \Theta_i)]\cdot(\tau_{ij\sigma}, -1)}{[r(\Theta_j, \Theta_i)]\cdot(\tau_{ij\sigma}, -1)}
\eqn{stij1}
\ee
for which $\tilde{\phi}(\tau_{ij\sigma})=s$. In particular, since \re{ful} reads
$\varrho\Theta_\sigma=[\phi\uu]-s[\uu]$
and   $\uu$ can be written as
 \be
 \uu=u_{j\sigma}(\tau_{1j\sigma},\dots,\tau_{jj\sigma},\dots, \tau_{ij\sigma},\dots, \tau_{nj\sigma}),\quad i>j,
 \eqn{tautau}
 \ee
 equation \re{ful} can be expressed as

 \be
 \varrho\Theta_\sigma=([\phi u_j]-s[u_j])(\tau_{1j\sigma},\dots,1 ,\dots, \tau_{ij\sigma},\dots, \tau_{nj\sigma})
 \eqn{fultau}
 \ee
 where
 \be
 \tau_{kj\sigma}=\frac{[\phi u_k]-s[u_k]}{[\phi u_j]-s[u_j]}.
\eqn{taubit}
\ee
\subsubsection{The Linear Case}
\noindent If  $\tilde{\phi}$ is linear in $\tau_{ij}$, say $\tilde{\phi}(\tau_{ij\sigma})=\tilde{\tau}_{ij\sigma}=u_i/u_j$,  then  \re{taubit} becomes  quadratic for $k=i$, with  $s=\tau_{ij\sigma}$ satisfying
\be
s=\frac{[(u_i^2/u_j, u_i )]\cdot(1, -s)}{[(u_i, u_j)]\cdot(1, -s)} , \quad (u_i/u_j)_->s>(u_i/u_j)_+,
\eqn{lin}
\ee
from which one obtains
\be
u_{j-}(u_{i-}/u_{j-}-s)^2
=u_{j+}(s-u_{i+}/u_{j+})^2.
\eqn{quad2}
\ee		
This leads to admissible $s$ being given by
\be
s=\frac{u_{i+}|u_{j-}|^{1/2}+u_{i-}|u_{j+}|^{1/2}}{u_{j+}|u_{j-}|^{1/2}+u_{j-}|u_{j+}|^{1/2}},
\eqn{lins}
\ee
provided $u_{j-}u_{j+}>0$, and then
\re{fultau} reduces to
\be
\varrho\Theta_\sigma=([u_i]-s[u_j])(\tau_{1j\sigma},\dots,1 ,\dots, s,\dots, \tau_{nj\sigma})
\eqn{taulin}
\ee
where
\be
\tau_{kj\sigma}=\frac{[\frac{u_iu_k}{u_j}]-s[u_k]}{[u_i]-s[u_j]}, \,\,k\neq i.
\eqn{bitlin}
\ee
This reduces  further to
\be
\varrho\Theta_\sigma=sgn\{u_{j\pm}\}[\frac{u_i}{u_j}]\sqrt{u_{j-}u_{j+}}\,(\varsigma_{j}(u_1),\dots,1 ,\dots, s,\dots, \varsigma_{j}(u_n))
\eqn{lin}
\ee
where 
\be
\varsigma_j(u_k)=\frac{u_{k+}\sqrt{|u_{j-}|}+u_{k-}\sqrt{|u_{j+}|}}{u_{j+}\sqrt{|u_{j-}|}+u_{j-}\sqrt{|u_{j+}|}},\, k\neq i.
\eqn{wherelin}
\ee

The next examples involve linear $\tilde{\phi}$, and follow from the properties stated above.

\subsubsection{Pressureless  Gas -  Relativistic and Nonrelativistic Cases}
\label{gas}

Finally, we examine the 3x3 system of equations describing a pressureless gas. Taken either in  its relativistic, or in its limiting nonrelativistic form, the basic structure is formally identical, stemming from the local conservation laws in the general-relativistic hydrodynamics equations, which involve the stress-energy tensor $T^{\mu\nu}$ and matter density current $J^{\mu}$, (\cite{Font}, \cite{LSZ}).

First we give a brief summary of the relativistic setting (for one spatial dimension) and then its asymptotic limit (as the speed of light tends to infinity) to obtain the corresponding nonrelativistic equations.

Introducing $x^0=ct, x^1=x$, one defines space-time displacements via $ds^2=-(dx^0)^2+(dx^1)^2$ as timelike if $ds^2<0$, in which case the proper time, $\tau$, is given by $c^2d\tau^2=-ds^2$. Then $d\tau=dt(1-\beta^2)^{1/2}$, where $\beta=\frac{1}{c}\frac{dx}{dt}\equiv\frac{v}{c}$. Letting 
$\Gamma=(1-\beta^2)^{-1/2}$, set $\rho_0$ to represent  rest mass density, and $\rho=\rho_0\Gamma$. Also set $e=\rho_0c^2+\epsilon$ to denote  the relativistic enthalpy, which contains both the rest-mass energy and $\epsilon\geq 0$, the  internal energy density of the gas in its rest frame.

Additionally, suppose $v^\mu(t)=\frac{dx^\mu}{dt}, \mu=0, 1,$ and $U^\mu(t)=\frac{1}{c}\frac{dx^\mu}{d\tau}=\frac{\Gamma}{c}v^\mu$, so that $(U^0, U^1)=\Gamma(1, \beta)$.   The current, $J^\mu$, now  satisfies the continuity equation
$$
\frac{\partial}{\partial x^\mu}(\rho_0U^\mu)=0,
$$
or
$$
\frac{1}{c}\frac{\partial}{\partial t}(\rho_0\Gamma)+\frac{\partial}{\partial x}(\rho_0\Gamma\frac{v}{c})=0,
$$
and   therefore 
\be
\rho_t+(\rho v)_x=0.
\eqn{J}
\ee
Meanwhile, the stress-energy tensor, $T^{\mu\nu}$, has components 
$$T^{00}=\Gamma^2 e, T^{01}=T^{10}=\Gamma^2 e\beta, T^{11}=\Gamma^2 e\beta^2$$
which satisfy the Bianchi identities
$$\frac{\partial}{\partial x^\mu}T^{\mu\nu}=0, \quad \nu=0, 1.$$
Setting $E=\Gamma^2e$ therefore gives
$$
\frac{1}{c}\frac{\partial}{\partial t}E+\frac{\partial}{\partial x}(E\beta)=0,
$$
and
$$
\frac{1}{c}\frac{\partial}{\partial t}(E\beta)+\frac{\partial}{\partial x}(E\beta^2)=0,
$$
from which one obtains simply
\be
E_t+(Ev)_x=0,
\eqn{T0}
\ee
and
\be
(Ev)_t+(Ev^2)_x=0.
\ee
In our setting, we   choose as variables
\be
u_1=E,\quad u_2=Ev,\quad u_3=\rho,
\eqn{3x3}
\ee
with $\tilde{\phi}=\tau_{21}=v$. Then \re{lin}, \re{wherelin} become
\be
\varrho\Theta_\sigma=[v]\sqrt{E_-E_+}\,(1 , s, \frac{\rho_+\sqrt{E_-}+\rho_-\sqrt{E_+}}{\sqrt{E_-}+\sqrt{E_+}}),
\eqn{gasrel}
\ee
with
\be
s=\frac{v_-\sqrt{E_-}+v_+\sqrt{E_+}}{\sqrt{E_-}+\sqrt{E_+}}.
\eqn{sgasrel}
\ee

To obtain correspondence with classical pressureless gas equations one  assumes $\beta$ to be small, \lq$c\rightarrow\infty$\rq,  in order to approximate $\Gamma$ by $1+\frac{1}{2}\beta^2$ and $\Gamma^2$ by $1+\beta^2$. Substituting these in the above equations, and dropping higher order terms in $\beta^2$ later,
$T^{00}=\Gamma^2e=\Gamma\rho c^2+\Gamma^2e$ is approximated by $(1+\frac{1}{2}\beta^2)\rho c^2+\epsilon$, similarly with $T^{10}=\Gamma^2e\beta$ and 
$T^{11}=\Gamma^2e\beta^2$. 
Then, the approximation to $\frac{\partial}{\partial x^{\mu}}T^{\mu 1}=0$  (with $\rho_0$ and $\rho$ now identified) gives
$$
(\rho v)_t+(\rho v^2)_x=0,
$$
while using the retained continuity equation \re{J} to simplify $\frac{\partial}{\partial x^{\mu}}T^{\mu 0}=0$    leads  to
$$
(\frac{\rho v^2}{2}+\epsilon)_t+((\frac{\rho v^2}{2}+\epsilon)v)_x=0.
$$
 Once again relabelling, we now set 
 \be
u_1=\rho, \quad u_2=\rho v, \quad u_3 = \frac{\rho v^2}{2}+\epsilon={\mathcal E}
\eqn{3x3n}
\ee
in order  to have $\tilde{\phi}=\tau_{21}=v$, as before. The relativistic results  \re{gasrel} and \re{sgasrel} are then replaced with (\ct{YZ2})

\be
\varrho\Theta_\sigma=[v]\sqrt{\rho_-\rho_+}\,(1 , s, \frac{{\mathcal E}_+\sqrt{\rho_-}+{\mathcal E}_-\sqrt{\rho_+}}{\rho_+\sqrt{\rho_-}+\rho_-\sqrt{\rho_+}}),
\eqn{gasnon}
\ee
with
\be
s=\frac{v_-\sqrt{\rho_-}+v_+\sqrt{\rho_+}}{\sqrt{\rho_-}+\sqrt{\rho_+}}.
\eqn{sgasnon}
\ee
In both cases, it follows by \re{lin}, that $c>v_->s>v_+>-c$.\qed
\begin{remark}
As pointed out by Shelkovich et al., in \cite{NilssonShelkovich2011_ApplAnal_I}, \cite{NilssonRozanovaShelkovich2011_ApplAnal_II}, 
in the absence of an internal energy term, $\epsilon$, in \re{3x3n}, the equations are overdetermined and  nontrivial delta-shock solutions    do not generally  exist. 
\end{remark}

 \begin{appendix}
 \section{Projection}\label{A}
 
  Equation \re{fin4} is stated purely in terms of $\tau_\sigma$, with typical physical quantities such as density and momentum  being represented   coordinate-axis variables in phase space.  Velocity can be  represented in terms of their  ratio. For example, in zero-pressure gas dynamics,  where $u_1=\rho, u_2=\rho u$,     velocity is given by $u=u_2/u_1\equiv\tau_{21}$, which need not be the same as $\tau$ if $n>2.$ 
 With this in mind when $n\geq 3$, we focus on the  case where $\phi$ depends only on a projection, $P_{ij}\Theta(\gamma), i\neq j$, and $\tilde{\phi}$  on the ratio $\tau_{ij}=\Theta_i/\Theta_j=u_i/u_j$, each defined in terms of the standard  basis, $\{{\mathbb E_k}\}_{k=1}^{n}\in{\mathbb R^n}$.
 The projection of $\Theta(\gamma)$ in \re{fin4} onto the $ij$ plane is the nonzero curve
\be
P_{ij}\Theta(\gamma)=P_{ij} ({(1+\tau^2)^{-1/2}}({\bf e}+ \tau{\bf f}))
= (1+\tau^2)^{-1/2}((e_i+f_i\tau){\mathbb E_i}+(e_j+f_j\tau){\mathbb E_j}),
\eqn{pox}
\ee
 where we assume ${e_j} {f_i}\neq{e_i}{f_j}$ for the projections of ${\bf e}$ and ${\bf f}$  to be linearly independent, and we note that  

\be
\tau_{ij}=
\frac{{e_i}+ \tau{f_i}}{{e_j}+\tau{f_j}}, \quad i\neq j,
\eqn{fin5}
\ee
has inverse  given by

\be
\tau=\frac{\tau_{ij}e_j-e_i}{f_i-\tau_{ij}f_j}.
\eqn{fin6}
\ee  
Further, if
 $\phi$ is to depend only on $P_{ij}\Theta$, then
\be
\tilde{\phi}(\tau)=\phi(\Theta(\gamma))=\phi(P_{ij}\Theta(\gamma))=\tilde{\phi}(\tau_{ij})\circ\tau.
\eqn{tox}
\ee 
 Therefore $\tilde{\phi}'(\tau)=\tilde{\phi}'(\tau_{ij})\circ\tau\,\frac{d\tau_{ij}}{d\tau}=\tilde{\phi}'(\tau_{ij})\circ{\tau}\,\frac{e_jf_i-e_if_j}{(e_j+\tau f_j)^2}$
and monotonicity of $\tilde{\phi}(\tau_{ij})$ follows from that for $\tilde{\phi}(\tau)$ provided $e_jf_i>e_if_j$, with ordering $\gamma_+>\gamma_-\iff\tau_+>\tau_-\iff\tau_{ij+}>\tau_{ij-}$ provided  $\tilde{\phi}'(\tau)$ is of constant sign.

 Finally, taking quotients in \re{ful}  gives
 \be
 \tau_{ij\sigma}=\frac{\alpha_+\Theta_{i+}+\alpha_-\Theta_{i-}}{\alpha_+\Theta_{j+}+\alpha_-\Theta_{j-}}
 \eqn{more}
 \ee
 where $\alpha_\pm=\pm(\tilde{\phi}(\tau_\sigma)-\phi_{\pm})r_{\pm}$, and so  $\tau_{ij\sigma}$ needs to satisfy
\be
\tau_{ij\sigma}=\frac{(\tilde{\phi}(\tau_{ij\sigma})-\phi_{+})r_{+}\Theta_{i+}-(\tilde{\phi}(\tau_{ij\sigma})-\phi_{-})r_{-}\Theta_{i-}}{(\tilde{\phi}(\tau_{ij\sigma})-\phi_{+})r_{+}\Theta_{j+}-(\tilde{\phi}(\tau_{ij\sigma})-\phi_{-})r_{-}\Theta_{j-}}
=\frac{[\phi r\Theta_i]-\tilde{\phi}(\tau_{ij\sigma})[r\Theta_i]}
{[\phi r\Theta_j]-\tilde{\phi}(\tau_{ij\sigma})[r\Theta_j]}
\eqn{tij}
\ee
or, as in \re{stij1},
\be
\tilde{\phi}(\tau_{ij\sigma})=\frac{[\phi r\Theta_j]\tau_{ij\sigma}-[\phi r\Theta_i]}
{[r\Theta_j]\tau_{ij\sigma}-[r\Theta_i]}\equiv\frac{[\phi(u_j, u_i)]\cdot(\tau_{ij\sigma}, -1)}{[(u_j, u_i)]\cdot(\tau_{ij\sigma}, -1)},
\eqn{stij}
\ee
in order to find $\tilde{\phi}(\tau_{ij\sigma})=s$. 
 \footnote{This can also be seen   substituting \re{fin6} into \re{fin4} at $\gamma_\sigma$, from which $\tau_\sigma$ follows via \re{fin6}. Obtaining uniqueness of $\Theta_\sigma\in sp\{\Theta_-, \Theta_+\}$ is  equivalent to that of finding uniqueness for its projection in $sp\{P_{ij}\Theta_-, P_{ij}\Theta_+\}$, by the comments after\re{tox}.}.

\end{appendix}


\section{Bibliography}

\end{document}